\documentclass[accepted]{uai2026} 
                        
\usepackage[american]{babel}

\usepackage{natbib} 
\usepackage{mathtools} 
\usepackage{booktabs} 
\usepackage{tikz} 
\usepackage{algorithm}
\usepackage{algpseudocode}
\usepackage{amssymb}
\usepackage{amsthm}
\usepackage{multicol}
\usepackage{subcaption}

\DeclarePairedDelimiter{\set}{\lbrace}{\rbrace}

\newtheorem{theorem}{Theorem}[section]
\newtheorem{lemma}[theorem]{Lemma}

\theoremstyle{definition}
\newtheorem{definition}{Definition}[section]
\theoremstyle{remark}
\newtheorem{remark}[definition]{Remark}
\newtheorem{example}[definition]{Example}

\title{Sign Identifiability of Causal Effects in Stationary Stochastic Dynamical Systems}

\author[1]{\href{mailto:<g.g.van.seeventer@liacs.leidenuniv.nl>?Subject=UAI 2026 paper}{Gijs~van~Seeventer}{}}
\author[1]{\href{mailto:<s.salehkaleybar@liacs.leidenuniv.nl>?Subject=UAI 2026 paper}{Saber~Salehkaleybar}{}}
\affil[1]{%
    LIACS\\
    Leiden University\\
    Netherlands
}

\begin{document}
\maketitle

\begin{abstract}
  We study identifiability in continuous-time linear stationary stochastic differential equations with a known causal structure. Unlike existing approaches, we relax the assumption of a known diffusion matrix, thereby respecting the model’s intrinsic scale invariance. Therefore, rather than recovering drift coefficients themselves, we introduce edge-sign identifiability: for a given causal structure,  we ask whether the sign of a given drift entry is uniquely determined across all observational covariance matrices induced by parametrisations compatible with that structure. This leads to a trichotomy of edge-sign identifiability: identifiable, non-identifiable, and partially identifiable. This trichotomy introduces the new notion of partial identifiability to the literature, which we show is a genuine category in our setting. Under a notion of faithfulness, we derive criteria to identify membership of each category for general graphs. Applying our criteria to specific causal structures, both analogous to classical causal settings (e.g., instrumental variables) and novel cyclic settings, we determine their edge-sign identifiability and, in some cases, obtain explicit expressions for the sign of a target edge in terms of the observational covariance matrix.
\end{abstract}

\section{Introduction}\label{sec:intro}
Learning dynamical systems from observational data via parametrised models is central across scientific domains, from systems biology \citep{marbach2012wisdom} to economics \citep{hamilton2020time}. The aim of such modelling is often to answer a causal question: \textit{what happens to a target variable $Y$ if we intervene on a variable $X$?} 

Recently, causal modelling of stationary diffusions has emerged to address settings in which time trajectories are unavailable or unobservable \citep{fitch2019learning,lorch2024causal}. In such cases, observations can be viewed as samples collected from a stationary process. These processes are often described by stationary stochastic differential equations (SDEs)
\citep{ksendal2003}. While stationary SDEs induce time-invariant observational distributions, they internally encode temporal causal dependencies, allowing for a natural causal interpretation \citep{lorch2024causal,amendola2025structural}. This interpretation admits a graphical representation analogous to structural causal models (SCMs) \citep{sokol2014causal,Pearl2009-fq}. Unlike acyclic SCMs, graphical models induced by stationary SDEs naturally allow for cycles and self-loops, features ubiquitous in real-world dynamical systems.

In the literature on graphical models, it is often assumed that the causal structure, namely, which variables directly affect others, is known. Given a fixed causal structure, a central question is identifiability, i.e., whether the value of a causal effect (or a property of it, such as its sign) can be determined from the observational distribution. While determining the magnitude of a causal effect is preferable, the sign alone is often informative, e.g., whether an intervention helps or harms. Therefore, sign identifiability has been and remains an active research topic \citep{manski1990nonparametric,lanners2026data}. 

Identifiability is also a central question in the literature on causal SDE models, which studies causal SDEs under various assumptions (see Section~\ref{sec:related_work}). In the absence of path data, we focus on the most tractable and interpretable SDE: continuous-time, linear, time-homogeneous stationary SDEs, i.e., the stationary Ornstein-Uhlenbeck (OU) process. For OU processes, existing notions of identifiability assume a known causal structure and a fixed part of the model parameters, specifically the diffusion matrix. However, fixing the diffusion matrix imposes a strong restriction since the OU process is invariant under positive rescaling (see Section~\ref{sec:background_sde}). Therefore, respecting this scale invariance, we assume that only the causal structure is known, while the diffusion matrix is allowed to be rescaled. Our main contributions are as follows:

\begin{itemize}
    \item \textbf{Sign identifiability.} \quad
     Only assuming the causal structure to be known limits what can be learned about the causal effect. Therefore, we introduce a notion of edge-sign identifiability (see Section~\ref{sec:definitions}), focusing on the sign rather than the magnitude of a causal effect. 
    \item \textbf{Categories of sign identifiability.}\quad
    Our analysis distinguishes three categories: identifiable, non-identifiable, and partially identifiable. While the term \emph{partial identifiability} is often used as a name for the broader research area of bounding causal effects \citep{lanners2026data}, we use it here to denote a new identifiability category for edge signs. We also show that in the confounding structure, the sign of the causal effect is partially identifiable with positive measure (for more details, see Section~\ref{sec:general_sign_identifiability}), making it a genuine category. Moreover, numerical experiments in Section~\ref{sec:numerical_results} suggest that partial identifiability constitutes an important intermediate category besides identifiability and non-identifiability.
    \item \textbf{General criteria and applications to specific structures.} \quad
    We derive general criteria for edge-sign identifiability and apply them to specific graph structures in Section~\ref{sec:edg_sign_identifiability_results}. General criteria include, for instance, a graphical criterion for determining whether an edge is identifiable. We apply the general criteria in some special graph structures, including causal structures analogous to bivariate cause–effect and instrumental variable settings, for which we obtain an explicit expression for the sign of the causal effect in terms of the covariance matrix.
\end{itemize}

\section{Preliminaries and Problem Setup}
\subsection{Background and Notation}
We denote random vectors in $\mathbb{R}^d$ as $X\in\mathbb{R}^d$. Its $i$th entry is $X_i\in\mathbb{R}$ for $i\in\{1,\dots,d\}$. In addition, we denote the set of positive definite matrices  as $PD_d$ and the set of positive definite diagonal matrices as $PDD_d$, where $d$ is the dimension. Note that $PDD_d\subset PD_d$.

\subsubsection{Stochastic Differential Equations}\label{sec:background_sde}
Stochastic differential equations (SDEs) describe stochastic processes $\{X(t)\}, X(t)\in\mathbb{R}^d$, which is a collection of random vectors $X(t)$ for a time $t$. We will consider \textit{continuous} \textit{stationary} processes, which means that the probability density $f_t\big(X(t)\big)$ is the same for all considered times, i.e., $f_t\big(X(t)\big)=f_{t'}\big(X(t')\big)$ with $t,t'\geq 0$. A general SDE has the form $dX(t)=g(X(t),t)dt+h(X(t),t)d\beta(t)$, where $g:\mathbb{R}^d\times\mathbb{R}_{\ge 0}\rightarrow\mathbb{R}^d$ is the drift function, $h:\mathbb{R}^d\times\mathbb{R}_{\geq 0}\rightarrow\mathbb{R}^{d\times d}$ is the diffusion function and $\beta$ is a random noise term.  We consider linear, time-homogeneous SDEs driven by a Wiener process, i.e., SDEs with linear, time-invariant drift and diffusion functions and Gaussian noise. We will assume the commonly used $It\hat{o}$ interpretation, which is a choice in what it means to integrate a random noise term \citep{ksendal2003}.

The only non-trivial continuous stationary, linear and time-homogenous SDE is known as the multivariate Ornstein-Uhlenbeck (OU) process \citep{doob}.\footnote{If the OU process is not initialized at its stationary distribution, it converges to stationarity only asymptotically. We therefore assume that the process is either initialized at stationarity or observed after a sufficiently long period. After this period, we relabel time so that the stationary regime starts at $t=0$ and consider $t\ge 0$.}
The OU process is described by
\begin{equation}\label{eq:OU}
    dX(t)=A\Big(X(t)-b\Big)dt + CdW(t),
\end{equation}
where $A\in\mathbb{R}^{d\times d}$ is the drift matrix, $b\in \mathbb{R}^d$ is a constant vector, $C\in\mathbb{R}^{d\times d}$ is the diffusion matrix and $W(t)$ is the $d$-dimensional Wiener process at time $t$. Due to being a Wiener process, the OU process is characterised by the first two moments, i.e., the mean $m(t)$ and the covariance matrix $\Sigma(t)$. Since we consider a stochastic \emph{stationary} process, $\frac{d}{dt}m(t)=0$ and $\frac{d}{dt}\Sigma(t)=0$. Therefore the mean $m=b$ and every entry of the covariance matrix $\Sigma_{ij}=\mathbb{E}\big[(X_i(t)-m)(X_j(t)-m)\big]$.  Without loss in generality, we take the mean $m=0$ such that $X(t)\sim\mathcal{N}(0,\Sigma)$.  Moreover, the stationarity condition $\frac{d}{dt}\Sigma=0$ is equivalent to the Lyapunov equation
\begin{equation}\label{eq:lyapunov}
    A\Sigma+\Sigma A^T = -D,
\end{equation}
where $D=CC^T$ \citep{Sarka_Solin_2019}. In addition, stationarity requires the drift matrix $A$ to be Hurwitz stable, i.e., the real parts of the eigenvalues of $A$ must be negative. Note that, given $D\in PD_d$, we have that $A$ is Hurwitz stable in the Lyapunov equation (Eq.~\eqref{eq:lyapunov}) if and only if $\Sigma\in PD_d$ \citep[Theorem~1.1]{frommer2012verified}.  Furthermore, if we choose an uncorrelated noise setting, i.e., the diffusion matrix $C$ being diagonal, we will always have $D\in PDD_d\subset PD_d$. If we have correlated noise, $C$ is no longer diagonal, and it needs to be verified that $D\in PD_d$. Unless otherwise noted, we will assume $D\in PDD_d$.

A Wiener process $W$ is scale invariant, i.e., $W(t)=W(at)/\sqrt{a}$ for $a\in\mathbb{R}_{+}$. Furthermore, given that we are considering a stationary process, we have that two times $t,at\in[0,\infty)$ are indistinguishable in the sense that $X(t)\overset{d}{=}X(at)$ in distribution. This means that we can always rescale the OU process with $\sqrt{a}$ to rewrite the original Eq.~\eqref{eq:OU} into
\begin{equation}
    dX(t)=aA\big(X(t)-m\big)dt + \sqrt{a}CdW(t).
\end{equation}
Note that scaling $C\mapsto\sqrt{a}C$ implies $D=CC^T\mapsto aD$. To preserve the Lyapunov equation (Eq.~\eqref{eq:lyapunov}), we need that $A\mapsto aA$. This means that we can always rescale the $(A,D)\mapsto(aA,aD)$. Therefore, from the covariance matrix, we can only identify the parameters of the drift matrix $A$ up to some global scaling $a>0$.

\subsubsection{Graphs}
Let $G=(V,E)$ be a directed graph with node set $V=\{V_0,\dots,V_d\}$ and edges $E$ where an edge $e:=(V_i,V_j)\in E$ is a directed edge from $V_i$ to $V_j$. Unless stated otherwise, all graphs are directed.

A directed path $V_i,\dots,V_j$ is a sequence such that there is a directed edge $V_k\rightarrow V_{k+1}$ for all $k=i,\dots,j-1$, where we call $V_i$ the ancestor of $V_j$. We denote the set of ancestors of $V_i$ in graph $G$ by $\operatorname{An}_G(V_i)$. Furthermore, if a directed path starts and ends at the same node, we call it a cycle, i.e. $V_i\rightarrow \dots \rightarrow V_i$. If a cycle has $n$ distinct nodes, then the cycle has length $n$. A special example of a cycle is the self-loop. A self-loop is a directed path from a node to itself, i.e., $V_i\rightarrow V_i$. Note that if a node $V_i$ is in a cycle, it is in its own ancestral set, i.e., $V_i\in\operatorname{An}_{G}(V_i)$.  If there are no cycles in a graph, we call it a directed acyclic graph (DAG). 

\subsubsection{Causal Interpretation of SDEs}\label{sec:background_causal}
Stochastic processes described by SDEs admit a causal interpretation in which interventions correspond to modifications of the equations \citep{sokol2014causal, lorch2024causal}. In particular, for the OU process (Eq.~\eqref{eq:OU}), we interpret a non-zero drift matrix entry $A_{ij}\neq0$ as a direct causal effect of process $X_j$ on $X_i$, with a causal strength given by $A_{ij}$, in line with \citep{amendola2025structural,varando2020graphical}. Under this perspective, several concepts from the structural causal model (SCM) literature are useful. One such concept is the representation of direct causal relations by a directed graph $G=(V,E)$, where the nodes $V_i\in V$ correspond to the process $X_i$ and edges $e\in E$ to the presence of a direct causal effect. When considering only the orientation of direct causal effects, i.e., ignoring their strength, we refer to this as the causal structure \citep{peters2017elements}. Accordingly, a drift matrix entry $A_{ij}\neq0$ is represented by an edge $V_j\rightarrow V_i$ with edge weight $A_{ij}$. With slight abuse of notation, we will represent an edge $e=(V_j,V_i)\in E$ with its corresponding drift entry $A_{ij}$, writing $A_e:=A_{ij}$. Then, $\operatorname{supp}(A)\subseteq E$ means that $ e\not\in E\implies A_{e}=0$. If equality holds, i.e., $\operatorname{supp}(A)=E$, then $ e\not\in E\iff A_{e}=0$. The latter corresponds to structural minimality in the SCM literature \citep{peters2017elements}. We will always assume structural minimality.

In addition, since Hurwitz stable matrices may have non-zero diagonal entries and need not be triangular, the graphs representing SDEs can contain cycles and self-loops. For triangular matrices, however, the eigenvalues coincide with the diagonal entries. Hence, Hurwitz stability in the triangular case requires all diagonal entries to be negative.\footnote{Ignoring diagonal entries (self-loops), a triangular matrix corresponds to a directed acyclic graph (DAG).}

A useful distinction in the SCM literature is between causal discovery (learning the causal graph from data) and causal inference (causal effect identification given a specified graph). We consider the latter setting and assume the directed graph $G=(V,E)$ is known.

Transferring concepts from the existing SCM literature to SDE-based models requires care because the underlying data-generating mechanisms differ. In an SCM, data are generated by structural assignments, and conditional-independence constraints are typically characterised via graphical separation criteria (e.g., $d$-separation for DAGs and $\sigma$-separation for cyclic SCMs) \citep{peters2017elements,bongers2021foundations}. In contrast, an SDE generates data through continuous-time stochastic dynamics, and in our setting, we observe samples from the stationary distribution induced by the SDE. Moreover, for SDEs describing stationary processes, such as the OU process, marginal independences are the only source of independence relations \citep{boege2025conditional}. This contrasts sharply with the SCM literature in which additional conditional independences arise and are characterised via $d$- and $\sigma$-separation.

\subsection{Definitions}\label{sec:definitions}
This section introduces the definitions used throughout the paper. We begin by defining a set of covariance matrices $\Sigma$ that encode the marginal independences implied by the assumed directed graph $G$. As discussed in the previous subsection, in our setting, marginal independences are the only source of independence constraints coming from a graph $G=(V,E)$. Under the assumption of a diffusion matrix $D\in PDD_d$ and $\operatorname{supp}(A)\subseteq E$, \cite[Theorem~1.2]{boege2025conditional}  gave a graphical criterion for such marginal independencies based on ancestral relationships. In particular,  $X_i\perp\!\!\!\perp X_j$ in any random vector $X$ distributed according to a stationary distribution respecting the causal structure, if and only if, $\operatorname{An}_G(V_i)\cap\operatorname{An}_G(V_j)=\emptyset$. 

\begin{definition}[T-Faithfulness]\label{def:m-faithful}
    We define the set $F_G$ of $t$-faithful covariance matrices $\Sigma$ for a graph $G=(V,E)$ as
    \begin{equation}     \begin{split}
        F_G :=\{&\Sigma:\,\Sigma\in PD_d;\, \forall V_i,V_j\in V,  \\ 
        &\operatorname{An}_G(V_i)\cap\operatorname{An}_G(V_j)=\emptyset \iff \Sigma_{ij}=0\}.
    \end{split} \end{equation}   
\end{definition}

\begin{remark}
    In other words, $t$-faithfulness assures that the marginal independences captured by $\Sigma$ are exactly those that can be read off from the graph by checking common ancestral relationships of pairs of variables. The latter can be encoded more densely represented in a trek graph (see \cite[below Definition~1.1]{boege2025conditional}), hence the term $ t$-faithful. The definition further guarantees that $\Sigma \in PD_d$, thereby ensuring valid solutions to the Lyapunov equation (Eq.~\eqref{eq:lyapunov}) (see Section~\ref{sec:background_sde}).
\end{remark}

The following four definitions focus on the possible $\operatorname{sign}(A_e)\in\{+,-,0\}$ associated with edges $e$ in the directed graph $G=(V,E)$ under the $t$-faithfulness assumption. Our focus on signs is motivated by the scaling invariance discussed in Section~\ref{sec:background_sde} where, for a given $\Sigma$, if $(A,D)$ satisfies the Lyapunov equation (Eq.~\eqref{eq:lyapunov}), then $(a A, a D)$ also satisfies it for any $a>0$. Hence, the drift matrix is only identifiable up to a global positive rescaling, and for any given edge $e$, we treat its sign as the primary information that can be recovered from $\Sigma$. The first two Definitions~\ref{def:edge_signature_set} and \ref{def:possible_edge_set} characterise which covariance matrices $\Sigma$ are compatible with a given $\operatorname{sign}(e)$ under a graph $G$. The third Definition~\ref{def:edge_identifiability} introduces a new notion of identifiability. When fixing the covariance matrix $\Sigma$, the third definition reduces to the fourth Definition~\ref{def:pointwise_edge_sign_identifiability}.

\begin{definition}[Edge Signature Set]\label{def:edge_signature_set}
    For a graph $G=(V,E)$ and edge $e\in E$, we define the edge signature set $\mathcal{M}^{k}_{G,e}$ as
    \begin{equation}     \begin{split}
        \mathcal{M}^{k}_{G,e}:=&\{\Sigma\in F_G\,: \exists A,D \text{ s.t. }A\Sigma + \Sigma A^T= -D;\, \\
        &D\in PDD_d; \text{supp}(A)= E;\, \text{sign}(A_e)=k\},
    \end{split} 
    \end{equation}
    where $k\in\{+,-\}$ and
    \begin{equation}     \begin{split}
        \mathcal{M}^{0}_{G,e}:=&\{\Sigma\in F_G\,: \exists A,D \text{ s.t. }A\Sigma + \Sigma A^T= -D;\, \\
        &D\in PDD_d; \text{supp}(A)\subset E;\, \text{sign}(A_e)=0\}.
    \end{split} \end{equation}
\end{definition}

\begin{remark}\label{remark:edge_signature_set_interpretation}
    We interpret this definition as all $t$-faithful covariance matrices $\Sigma$ that could generate a $\pm$ sign for edge $e\in E$ from graph $G$ when the drift matrix matches the causal structure of the graph $G=(V,E)$. While we focus on studying the $\mathcal{M}^{+}_{G,e}$ and $\mathcal{M}^{-}_{G,e}$, the $\mathcal{M}^{0}_{G,e}$ edge signature set will be a useful theoretical tool.
\end{remark}

Since we are only interested in covariance matrices compatible with the graph under structural minimalism, we define a set of possible covariance matrices.

\begin{definition}[Possible Set]\label{def:possible_edge_set}
   For a graph $G=(V,E)$ and edge $e\in E$, define the possible set as
   \begin{equation}
   \mathcal{M}^{p}_{G,e}:=\mathcal{M}^{+}_{G,e}\cup\mathcal{M}^{-}_{G,e}.
   \end{equation}
\end{definition}

Next, we introduce the notion of edge-sign identifiability.

\begin{definition}[Edge-Sign Identifiability]\label{def:edge_identifiability}
    The sign of edge $e\in E$ in graph $G=(V,E)$ with signature-sets $\mathcal{M}^{k}_{G,e}$ and $k\in\{+,-\}$, is:
    \begin{itemize}
        \item non-identifiable if $\mathcal{M}^{+}_{G,e}=\mathcal{M}^{-}_{G,e}$ while $\mathcal{M}^{-}_{G,e}\neq\emptyset$,
        \item partially identifiable if $\mathcal{M}^{+}_{G,e}\neq\mathcal{M}^{-}_{G,e}$ and $\mathcal{M}^{+}_{G,e}\cap\mathcal{M}^{-}_{G,e}\neq\emptyset$,
        \item identifiable if $\mathcal{M}^{+}_{G,e}\cap\mathcal{M}^{-}_{G,e}=\emptyset$ while $\mathcal{M}^{-}_{G,e}\neq\emptyset$ or $\mathcal{M}^{+}_{G,e}\neq\emptyset$.
    \end{itemize}
\end{definition}

\begin{remark}\label{remark:interpretation_edge_identifiability}
    Intuitively, Definition~\ref{def:edge_identifiability} formalizes whether the sign of the edge weight $A_e$ is determined by the covariance matrix.
    If $e$ is \emph{identifiable}, then all the covariance matrices fix the sign.   
    If $e$ is \emph{non-identifiable}, then no covariance matrix ever resolves the sign in the sense that whenever the covariance matrix is compatible with one sign, it is also compatible with the other. If $e$ is \emph{partially identifiable}, then there exist covariance matrices for which the sign is uniquely determined, and there exist others that are compatible with both signs. Considering the edge weight $A_e$ as the direct causal strength, the edge-sign identifiability thus formalises whether we can learn its sign value. 
\end{remark}

When considering a fixed covariance matrix $\Sigma \in \mathcal{M}^p_{G,e}$, Definition~\ref{def:edge_identifiability}  reduces to the following definition. For a proof, see Appendix~\ref{sec:appendix_pointwise_clarification}.

\begin{definition}[Pointwise Edge-Sign Identifiability]\label{def:pointwise_edge_sign_identifiability}
        Let $e\in E$ be an edge in graph $G=(V,E)$ with signature-sets $\mathcal{M}^{k}_{G,e}$ and $k\in\{+,-\}$. For a covariance matrix $\Sigma\in\mathcal{M}^p_{G,e}$, the sign of $e$ is:
    \begin{itemize}
        \item non-identifiable if $\Sigma\in\mathcal{M}^{+}_{G,e}\cap\mathcal{M}^{-}_{G,e}$,
        \item identifiable if $\Sigma\not\in\mathcal{M}^{+}_{G,e}\cap\mathcal{M}^{-}_{G,e}$.
    \end{itemize}
\end{definition}

\subsection{Sign Identification Problem}
Assume the data-generating process is an OU process (Eq.~\eqref{eq:OU}) with uncorrelated noise such that $D\in PDD_d$ under the assumption of structural minimality. Given a directed graph $G=(V,E)$ and an edge $e\in E$, determine whether the sign of $e$ is non-identifiable, partially identifiable, or identifiable according to the Definition~\ref{def:edge_identifiability}. If identifiable, determine its $\operatorname{sign}(e)$.

\section{Related Work}\label{sec:related_work}
The use of SDEs as a causal modelling tool is an active field of research. SDEs can model both dynamic and stationary processes. Many works focus on dynamic processes \citep{mogensen2018causal,stippinger2023causal,cinquini2025practical}. This requires access to sample paths (time trajectories), which is not always feasible in practice. For example, most single-cell RNA sequencing techniques destroy the cell being sampled \citep{liu2024resolving}. On the other hand, works on stationary processes in general do not require access to sample paths, although there are exceptions \citep{manten2024signature}. 
Without requiring a sample path, typically obtained via discrete measurements, most works (discussed below) focus on continuous-time models. However, there are exceptions that consider discrete time \citep{recke2026identifiability}.

The research on causal modelling of stationary processes with continuous time has so far mainly focused on the study of linear SDEs with Gaussian noise, i.e., on stationary OU processes. Since the main interest is the direct causal effects and the associated sparsity structure, the research has focused on the drift matrix $A$. The drift matrix $A$ is constrained by the Lyapunov equation (Eq.~\eqref{eq:lyapunov}), which may explain why some works on \textit{causal} stationary OU processes refer to it as \textit{graphical continuous Lyapunov models} (GCLM), first coined in \citep{varando2020graphical}. In this line of research, initial works focused on causal discovery \citep{fitch2019learning,dettling2024lasso}.

More recently, identifiability has received increased attention. 
\cite{dettling2023identifiability} introduces a notion of (generic) identifiability based on uniqueness: for a given graph $G$ and $D \in PD_d$, identifiability holds if all stable drift matrices $A$ are in one-to-one correspondence (almost surely) with covariance matrices $\Sigma$. 
While they derive results under $D \in PD_d$, stronger and more comprehensive characterisations are obtained under the stricter assumption $D \in PDD_d$. Their characterisations combine conditions derived from the covariance matrix $\Sigma$ with structural constraints given by the sparsity pattern of $A$. Building on this work, \cite{amendola2025structural} considers $D \in PDD_d$ and graphs $G$ that are acyclic when self-loops are ignored. They provide a graphical criterion for model equivalence together with a polynomial-time algorithm to decide whether a model is unique in a given equivalence class and whether two models are equivalent.

Our work differs from these approaches by relaxing the strong assumption that $D$ is known. 
Requiring a fixed $D$ in addition to the graph, ignores the scale invariance inherent to the stationary OU process, as reflected both in the Lyapunov equation (Eq.~\eqref{eq:lyapunov}) and in the driving Wiener process. 
For example, as a consequence, the identifiability notion in \citep{dettling2023identifiability} does not capture \textit{partial sign identifiability}: for a given graph, there may exist admissible covariance matrices $\Sigma$ for which the sign of an edge weight is identifiable, and others for which it is not. As an example, see Remark~\ref{remark:confounding_difference_fixed_difussion}.

Beyond linear SDEs, recent advances have addressed causal discovery for general drift and diffusion parametrisations in stationary continuous-time processes. \cite{lorch2024causal} propose a method based on a kernel objective that quantifies the deviation of an SDE parametrisation from empirical observations.  \cite{bleile2026efficient} improved this approach in terms of computational efficiency. Our sign identifiability results could be of interest for linear parametrisations learned by these methods to assess whether they can correctly recover edge signs in cases where the signs are identifiable. In addition, in such settings, these methods also do not account for the existence of partial and non-identifiable edges.

\section{Edge-Sign Identifiability Results}\label{sec:edg_sign_identifiability_results}
This section presents theoretical results on edge-sign identifiability. All proofs rely on the Lyapunov equation (Eq.~\eqref{eq:lyapunov}). Section~\ref{sec:general_sign_identifiability} establishes theorems that hold for arbitrary graphs $G$. Section~\ref{sec:classical_and_new_graph_structures} uses these theorems to analyse specific causal structures. We present results both with and without latent variables.

\subsection{Sign Identifiability in General Graphs}\label{sec:general_sign_identifiability}
We begin by presenting two lemmas and a theorem ($\mathcal{M}^0_e$ criterion) that establish sign identifiability results for a fixed covariance matrix $\Sigma$. We continue by presenting a theorem (graphical criterion) and a supporting lemma. The theorem establishes a graphical rule to determine identifiability for a given graph $G$ and edge $e$, i.e., all covariance matrices entailed by the graph $G$ in the possible edge set. This theorem is valid only for graphs without latent variables (see Remark~\ref{remark:no_latent_graphical_criterion}). We end by providing two examples of applying the graphical criterion to general structures.

\begin{lemma}\label{lem:m0_implications}
    Let $e \in E$ be an edge in a graph $G=(V,E)$. Then
    \begin{align}
        \Sigma\in \mathcal{M}^{+}_{G,e} \text{ and }\Sigma\in\mathcal{M}^{-}_{G,e} &\implies \Sigma\in \mathcal{M}^{0}_{G,e} \label{eq:m0_criterion_m+m-_implies_m0},\\
        \Sigma\in \mathcal{M}^{+}_{G,e}\text{ and }\Sigma\in\mathcal{M}^{0}_{G,e} &\implies \Sigma\in \mathcal{M}^{-}_{G,e}  \label{eq:m0_criterion_m+m0_implies_m-},\\
         \Sigma\in \mathcal{M}^{-}_{G,e}\text{ and }\Sigma\in\mathcal{M}^{0}_{G,e} &\implies \Sigma\in \mathcal{M}^{+}_{G,e}.  \label{eq:m0_criterion_m0m-_implies_m+}
    \end{align} 
\end{lemma}

\textit{Proof sketch.}\quad The proof proceeds analogously for all implications. We select a covariance matrix $\Sigma$ in the intersection of the two signature sets on the left-hand side of the implication, which also ensures $\Sigma \in F_G$. For this $\Sigma$, we obtain two Lyapunov equations (Eq.~\eqref{eq:lyapunov}). Each can be rescaled by an arbitrary scalar $a \in \mathbb{R}$, and their sum remains a valid Lyapunov equation. Such a rescaling, however, need not correspond to a valid OU model. We therefore choose the rescaling so that $D \in PDD_d$. To satisfy Definition~\ref{def:edge_signature_set} for the signature set on the right-hand side of the implication, we additionally ensure that the rescaling preserves $\operatorname{supp}(A) = E$ and yields the required $\operatorname{sign}(A_e)$. The full proof is provided in Appendix~\ref{sec:appendix_lemma_m0_implications}.

\begin{remark}
    We emphasise that the key mechanism in the proof is the scale invariance of the Lyapunov equation: the drift matrix $A$ and the diffusion matrix $D$ can be rescaled while preserving both the OU model and the induced covariance matrix $\Sigma$. Whereas existing approaches eliminate this freedom by fixing the scale, we exploit it. The rescaling is therefore not merely a nuisance but a structural feature that is utilised in our sign identifiability results.
\end{remark}

\bigskip

\begin{theorem}[$\mathcal{M}^{0}$ Criterion]\label{thm:m0-criterion}
    Let $e\in E$ be an edge in graph $G=(V,E)$ and let $\Sigma\in\mathcal{M}^{p}_{G,e}$. Then
    \begin{equation}     \begin{split}
        \Sigma\in \mathcal{M}^{0}_{G,e} &\iff e \text{ is non-identifiable for }\, \Sigma.
    \end{split} \end{equation}
\end{theorem}

\begin{proof}
    In the proof, we always refer to the same graph $G=(V,E)$. For brevity, we will suppress the subscript $G$ and only write $\mathcal{M}^{k}_{e}$. There are eight possible combinations of edge signature set memberships for a given $\Sigma$:
\begin{enumerate}
\begin{multicols}{2}
    \item $\mathcal{M}^{0}_e\cap\neg\mathcal{M}^{+}_e\cap\neg\mathcal{M}^{-}_e,$
    \item $\mathcal{M}^{0}_e\cap\neg\mathcal{M}^{+}_e\cap\mathcal{M}^{-}_e,$
    \item $\mathcal{M}^{0}_e\cap\mathcal{M}^{+}_e\cap\neg\mathcal{M}^{-}_e,$
    \item $\mathcal{M}^{0}_e\cap\mathcal{M}^{+}_e\cap\mathcal{M}^{-}_e,$
    \item $\neg\mathcal{M}^{0}_e\cap\neg\mathcal{M}^{+}_e\cap\neg\mathcal{M}^{-}_e,$
    \item $\neg\mathcal{M}^{0}_e\cap\neg\mathcal{M}^{+}_e\cap\mathcal{M}^{-}_e,$
    \item $\neg\mathcal{M}^{0}_e\cap\mathcal{M}^{+}_e\cap\neg\mathcal{M}^{-}_e,$
    \item $\neg\mathcal{M}^{0}_e\cap\mathcal{M}^{+}_e\cap\mathcal{M}^{-}_e$,
\end{multicols}
\end{enumerate}
where we use the shorthand notation $\neg\mathcal{M}^{k}_e:=\{\Sigma\in\mathcal{M}^p_e:\Sigma\not\in\mathcal{M}^{k}_e\}$. Since $\Sigma\in\mathcal{M}^{p}_{e}$, we have that $\Sigma\not\in\neg\mathcal{M}^{0}_e\cap\neg\mathcal{M}^{+}_e\cap\neg\mathcal{M}^{-}_e$ and $\Sigma\not\in\mathcal{M}^{0}_e\cap\neg\mathcal{M}^{+}_e\cap\neg\mathcal{M}^{-}_e$. Moreover, Eq.~\eqref{eq:m0_criterion_m+m-_implies_m0} in Lemma~\ref{lem:m0_implications} yields $\Sigma\not\in \neg\mathcal{M}^{0}_e\cap\mathcal{M}^{+}_e\cap\mathcal{M}^{-}_e$. Similarly, Eq.~\eqref{eq:m0_criterion_m+m0_implies_m-}, implies $\Sigma\not\in\mathcal{M}^{0}_e\cap\mathcal{M}^{+}_e\cap\neg\mathcal{M}^{-}_e$, and Eq.~\eqref{eq:m0_criterion_m0m-_implies_m+}, implies $\Sigma\not\in\mathcal{M}^{0}_e\cap\neg\mathcal{M}^{+}_e\cap\mathcal{M}^{-}_e$.
This leaves three possible edge signature set combinations for $\Sigma$,
\begin{itemize}
\begin{multicols}{2}
    \item $\neg\mathcal{M}^{0}_e\cap\neg\mathcal{M}^{+}_e\cap\mathcal{M}^{-}_e$
    \item $\neg\mathcal{M}^{0}_e\cap\mathcal{M}^{+}_e\cap\neg\mathcal{M}^{-}_e$
    \item $\mathcal{M}^{0}_e\cap\mathcal{M}^{+}_e\cap\mathcal{M}^{-}_e$
\end{multicols}
\end{itemize}
Among the remaining combinations, if $\Sigma\in\mathcal{M}^{0}_e$, then $\Sigma\in~\mathcal{M}^{+}_e\cap\mathcal{M}^{-}_e$. In addition, if $\Sigma\in\neg\mathcal{M}^{0}_e$, then $\Sigma\not\in\mathcal{M}^{+}_e\cap\mathcal{M}^{-}_e$. Therefore, based on the pointwise edge-sign identifiability Definition~\ref{def:pointwise_edge_sign_identifiability}, $e$ is non-identifiable for $\Sigma\in\mathcal{M}^p_{e}$ if and only if $\Sigma\in\mathcal{M}^{0}_e$. 
\end{proof}

\begin{lemma}\label{lem:extending_m0}
    If Definition~\ref{def:edge_signature_set} is modified by replacing $D \in PDD_d$ with $D \in PD_d$, then Lemma~\ref{lem:m0_implications} and Theorem~\ref{thm:m0-criterion} remain valid.
\end{lemma}

\textit{Proof sketch.}\quad We redefine the edge signature sets by allowing $D \in PD_d$ instead of $D \in PDD_d$ (see Definition~\ref{def:pd_edge_signature_set}). Since \cite{boege2025conditional} require $D$ to be diagonal, we can no longer use their result. Therefore, we impose $\Sigma \in PD_d$ rather than $\Sigma \in F_G$. The proof of Lemma~\ref{lem:m0_implications} is analogous, except that establishing the existence of a rescaling with $D \in PD_d$ is slightly more involved. With this redefinition and the corresponding extension of Lemma~\ref{lem:m0_implications} to $D \in PD_d$, the proof of Theorem~\ref{thm:m0-criterion} proceeds unchanged. The full proof is provided in Appendix~\ref{sec:appendix_extending_m0}.

To verify the $\mathcal{M}^{0}$ criterion (Theorem~\ref{thm:m0-criterion}) in practice, it is useful to note that testing membership of an edge signature set can be formulated as a linear optimisation problem. See Remark~\ref{remark:feasability_problem} for a more detailed description. In addition, the $\mathcal{M}^{0}$ criterion is also a powerful theoretical tool, as will be showcased in the proof of Theorem~\ref{thm:graphical_criterion}. Before we can proceed to the theorem, we introduce the following lemma, which is essential to its proof.

 \begin{lemma}\label{lem:edge_deletion_Fg}
Let $G=(V,E)$, $G'=(V,E')$ and $G''=(V,E'')$ with $E''\subset E'\subset E$, if $F_G \neq F_{G'}$ then  $F_{G}\neq F_{G''}$.
\end{lemma}

\textit{Proof sketch.}\quad Observe that the number of marginally independent pairs of nodes can only increase due to edge deletion. Therefore, if a graph $G'$ is a graph with deleted edges with respect to a graph $G$, then it can only have the same or more marginally independent pairs of nodes. Furthermore, the $t$-faithful sets $F_G$ and $F_{G'}$ are only equal if they have the same marginally independent pairs of nodes. Using these observations twice, once for $G$ and $G'$ and once for $G'$ and $G''$, we obtain Lemma~\ref{lem:edge_deletion_Fg}. For the detailed proof, see Appendix~\ref{app:proof_edge_deletion_Fg}.

\begin{theorem}[Graphical Criterion]\label{thm:graphical_criterion}
    Without latent variables, let $G=(V,E)$ be a graph with an edge $e\in E$. Define $G'=(V,E\setminus\{e\})$. Then the edge $e$ is identifiable if the entailed marginal independencies by $G$ and $G'$ differ, i.e., if there exist $V_i,V_j\in V$ such that 
    \begin{equation}
        \operatorname{An}_{G'}(V_i)\cap\operatorname{An}_{G'}(V_j)=\emptyset\neq\operatorname{An}_G(V_i)\cap\operatorname{An}_G(V_j).
    \end{equation}
\end{theorem}

\begin{proof}
Let $G=(V,E)$ be a graph with an edge $e$ and $G'(V,E\setminus\{e\})$ be a graph such that the entailed marginal independencies of the graphs $G$ and $G'$ are different. Then setting edge $e=0$, i.e., removing it from graph $G$, results in graph $G'$ and the edge signature set $\mathcal{M}^{0}_{G,e}$ (see Definition~\ref{def:edge_signature_set}). Let $\Sigma\in\mathcal{M}^{0}_{G,e}$, then by definition of $\mathcal{M}^{0}_{G,e}$, we have that $\Sigma\in F_G$ and $\Sigma$ satisfies the Lyapunov equation for some drift matrix $A$. Since this drift matrix $A$ used to generate $\mathcal{M}^{0}_{G,e}$ only requires $\operatorname{supp}(A) \subset E$, it follows that $\operatorname{supp}(A) \subseteq E'$. Starting with $\operatorname{supp}(A) = E'$, we know that the marginal independencies are not the same. Therefore, if $\Sigma$ is generated via $\operatorname{supp}(A) = E'$ then $\Sigma\not\in F_G$. Using Lem.~\ref{lem:edge_deletion_Fg}, we then know that for all graphs $G''=(V,E'')$ with $E''\subseteq E'$ we have that $F_G\neq F_{G''}$. Hence, for any $\Sigma$ generated via $\operatorname{supp}(A) \subseteq E'$ we have that $\Sigma\not\in F_G$. Therefore, by the definition of $\mathcal{M}^{0}_{G,e}$ requiring $\Sigma\in F_G$, we have that $\mathcal{M}^{0}_{G,e}=\emptyset$. Hence, using the $\mathcal{M}^0$-criterion (Theorem~\ref{thm:m0-criterion}), we have that for all $\Sigma\in F_G$, edge $e$ in $G$ is identifiable. Since $\mathcal{M}^{p}_{G,e}\subseteq F_{G}$, we have that $e$ is identifiable in graph $G$ and the proof is complete.
\end{proof}

\begin{remark}\label{remark:no_latent_graphical_criterion}
    If a graph $G$ contains latent variables, the covariance matrix can be written as  $\Sigma=\big[\Sigma_{hh},\,\Sigma_{ho};\, \Sigma_{oh},\,\Sigma_{oo}\big]$, where $o$ denotes observable and $h$ denotes hidden. In this case, the observed block $\Sigma_{oo}$ constrains only part of the full covariance matrix. The observed covariance induces the set $\Sigma_{\mathrm{set}}:=\{\Sigma'\in \mathcal{M}^p_{G,e};\, \Sigma'_{oo}=\Sigma_{oo}\}\,$ consisting of all covariance matrices compatible with $G$ that agree on the observable block. It may then occur, even if all $\Sigma'\in\Sigma_{\mathrm{set}}$ are identifiable, that $\Sigma_{\mathrm{set}}\cap\mathcal{M}^+_{G,e}\neq\emptyset$ and $\Sigma_{\mathrm{set}}\cap\mathcal{M}^-_{G,e}\neq\emptyset$. Therefore, in that scenario $\Sigma_{\mathrm{set}}$ is not restricted to a single sign.
    In the sense of Definition~\ref{def:edge_identifiability}, this implies non-identifiability. Hence, $\Sigma_{oo}$ alone, even together with the observation that each $\Sigma\in \mathcal{M}^p_{G,e}$ is identifiable, is insufficient to conclude that the edge $e$ is sign identifiable. For example, we refer to the proof in Appendix~\ref{sec:appendix_latent_cause_n_effect}.
\end{remark}

To show the applicability of the graphical criterion (Theorem~\ref{thm:graphical_criterion}), we provide two examples of general applications of the theorem.

\begin{example}[Directed Tree]
    Any directed tree graph $G=(V,E)$ has the property that while $G$ is connected, removing any edge $e\in E$ makes the graph disconnected. Therefore, the deletion of any edge $e$, will change the marginal independences. Hence, using the graphical criterion, we have that all edges are identifiable. For examples of directed trees, see Figures~\ref{fig:hy} and~\ref{fig:chain}.
\end{example}

\begin{example}[General Instrumental Variable]
    Let $G=(V,E)$ be a graph with distinct nodes $i,j,k\in V$. For any edge $i\rightarrow j\in E$, if there exists a node $k$ that is an ancestor of $i$ and it has only one directed path to $j$ (through $i$), then removing edge $i\rightarrow j$ will create a new marginal independence. Hence, by the graphical criterion, we have that the edge $i\rightarrow j$ is identifiable. Note that $k$ functions similarly to an instrumental variable, but in contrast to an instrumental variable (IV), it can have a distance of more than one edge to the node $i$. For standard IV examples, see Figure~\ref{fig:iv} and~\ref{fig:iv_cyclic}.
\end{example}

\subsection{Classical and Novel Graph Structures}\label{sec:classical_and_new_graph_structures}
In this section, we study edge-sign identifiability for specific causal graphs. These graphs are analogous to those common in the acyclic SCM literature (e.g., instrumental variable and confounding settings), as well as novel graphs that allow cycles beyond self-loops.
The graphs we study are shown in Fig.~\ref{fig:graphs}. Each variable has a self-loop, but this has been suppressed in the figures for readability.  We are always interested in the sign of the (red) edge $\alpha$ (also indicated in the figures). For the graphs in Figs.~\ref{fig:hy}--\ref{fig:iv_cyclic}, we provide theoretical guarantees on whether the sign of the red edge $\alpha$ is identifiable. The last three graphs, Fig.~\ref{fig:one_proxy}--\ref{fig:two_proxy_circulair}, are studied numerically in the next section.

Note that in the case of the instrumental variable (IV) and the (one) proxy variables, the edge of interest corresponds to the common edge of interest in the literature. 
In Section~\ref{sec:latent} (the latent variable case), we consider the variable $H$ in the graphs of Fig.~\ref{fig:graphs} to be hidden.

\begin{figure*}
\centering

\begin{subfigure}{0.18\textwidth}
    \centering
    \includegraphics[width=\linewidth]{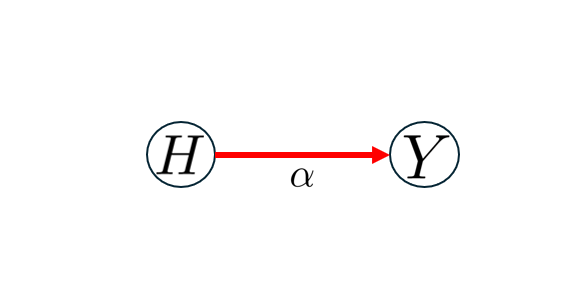}
    \caption{Cause and effect}
    \label{fig:hy}
\end{subfigure}
\hfill
\begin{subfigure}{0.18\textwidth}
    \centering
    \includegraphics[width=\linewidth]{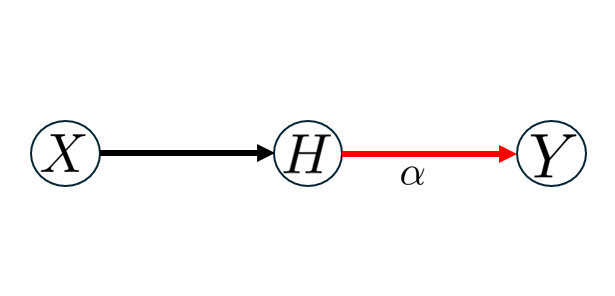}
    \caption{Chain}
    \label{fig:chain}
\end{subfigure}
\hfill
\begin{subfigure}{0.18\textwidth}
    \centering
    \includegraphics[width=\linewidth]{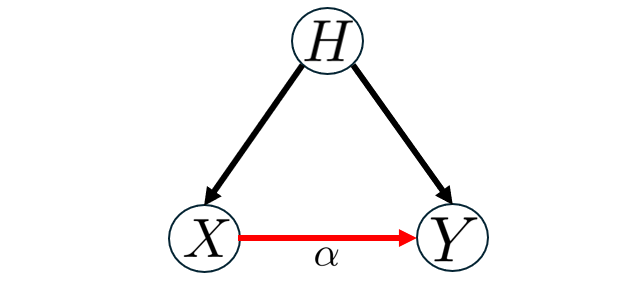}
    \caption{Confounding}
    \label{fig:confounding}
\end{subfigure}
\hfill
\begin{subfigure}{0.18\textwidth}
    \centering
    \includegraphics[width=\linewidth]{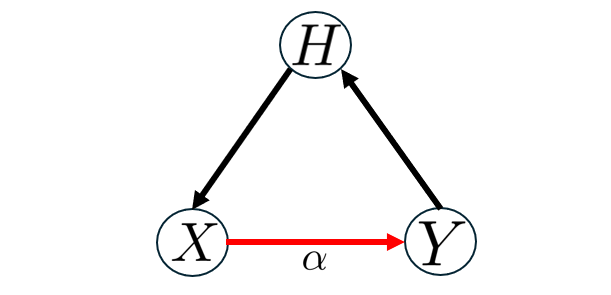}
    \caption{Cycle of length 3}
    \label{fig:three_loop}
\end{subfigure}

\begin{subfigure}{0.18\textwidth}
    \centering
    \includegraphics[width=\linewidth]{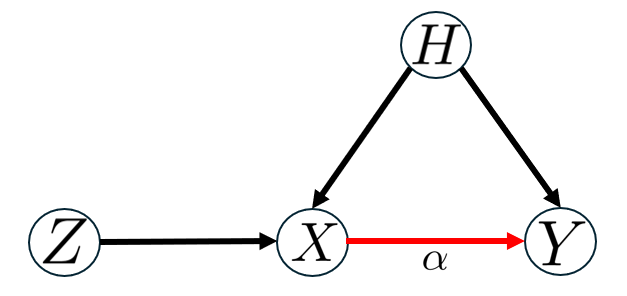}
    \caption{Instrumental variable}
    \label{fig:iv}
\end{subfigure}
\hfill
\begin{subfigure}{0.18\textwidth}
    \centering
    \includegraphics[width=\linewidth]{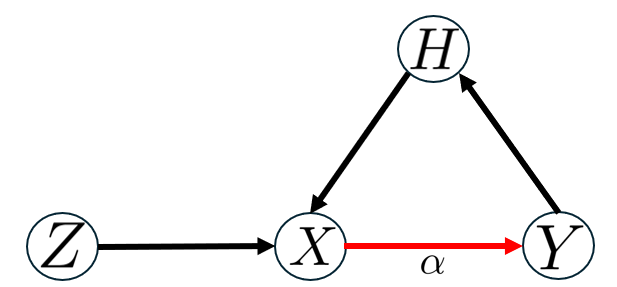}
    \caption{Cycle with IV}
    \label{fig:iv_cyclic}
\end{subfigure}
\hfill
\begin{subfigure}{0.18\textwidth}
    \centering
    \includegraphics[width=\linewidth]{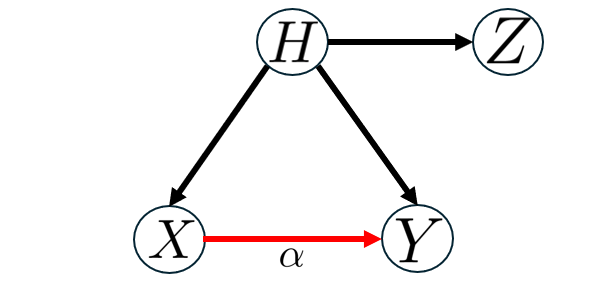}
    \caption{One proxy}
    \label{fig:one_proxy}
\end{subfigure}
\hfill
\begin{subfigure}{0.18\textwidth}
    \centering
    \includegraphics[width=\linewidth]{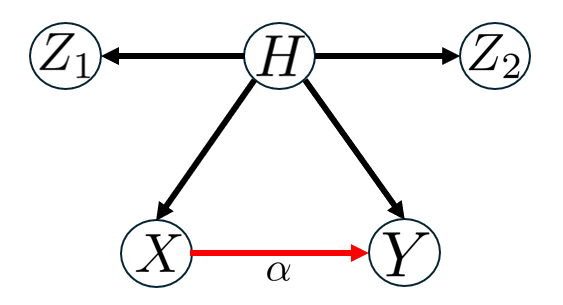}
    \caption{Two proxies}
    \label{fig:two_proxy}
\end{subfigure}
\hfill
\begin{subfigure}{0.18\textwidth}
    \centering
    \includegraphics[width=\linewidth]{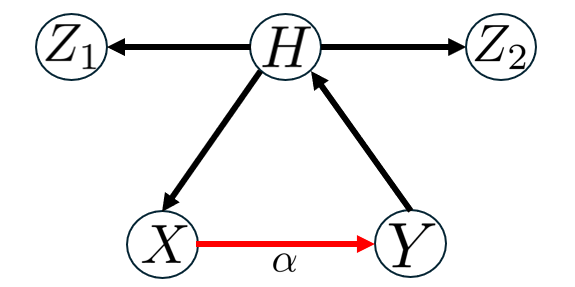}
    \caption{Cycle with proxies}
    \label{fig:two_proxy_circulair}
\end{subfigure}

\caption{Nine considered causal structures. The red edge $\alpha$ under consideration is indicated by $\alpha$. Figures (a)-(g) are discussed in both Section~\ref{sec:classical_and_new_graph_structures} and Section~\ref{sec:numerical_results}. The node $H$ is considered observable and latent in Section~\ref{sec:no_latent} and Section~\ref{sec:latent}, respectively. Figures~(h) and (i) are only studied numerically in Section~\ref{sec:numerical_results}.}
\label{fig:graphs}
\end{figure*}

\subsubsection{Without Latent Variables}\label{sec:no_latent}
In Theorem~\ref{thm:no_latent_id}, we characterize the edge-sign identifiability of $\alpha$ for the graphs in Fig.~\ref{fig:graphs} (Subfigures~\ref{fig:hy}--\ref{fig:iv_cyclic}) under $t$-faithfulness. Using the Lyapunov equation (Eq.~\eqref{eq:lyapunov}) and the $\mathcal{M}^0$-criterion (Theorem~\ref{thm:m0-criterion}), we derive algebraic constraints that characterise the sign identifiability in terms of covariance matrices $\Sigma \in \mathcal{M}^p_{G,\alpha}$. Specifically, this analysis separates the three identifiability categories (see Theorem~\ref{thm:no_latent_id}), specifies when a partially identifiable $\alpha$ becomes identifiable (see Lemma~\ref{lem:no_latent_part_id_values}), and yields an explicit formula for $\operatorname{sign}(\alpha)$ as a function of $\Sigma$ (see Lemma~\ref{lem:no_latent_id_values}) for some causal graphs. The drawback of using the $\mathcal{M}^0$-criterion is that it can be algebraically involved to show whether $\Sigma \in \mathcal{M}^{0}_{G,e}$. 

Theorem~\ref{thm:identifiability} establishes identifiability of $\alpha$ for the same graphs via a purely graphical argument. In contrast to Theorem~\ref{thm:no_latent_id}, it neither distinguishes between non- and partial identifiability nor characterises the conditions in terms of $\Sigma \in \mathcal{M}^p_{G,\alpha}$. However, its proof follows directly from the graphical criterion (Theorem~\ref{thm:graphical_criterion}) and is therefore substantially simpler than the proof of Theorem~\ref{thm:no_latent_id}.

\begin{theorem}[Edge-Sign Identifiability without Latent Variables]\label{thm:no_latent_id}
    For the red edge $\alpha$ in the graphs in Figs.~\ref{fig:hy} up to~\ref{fig:iv_cyclic} under the $t$-faithfulness assumption, the sign of $\alpha$ is:
    \begin{itemize}
        \item partially identifiable for~\ref{fig:confounding} and~\ref{fig:three_loop},
        \item identifiable for~\ref{fig:hy},\ref{fig:chain},\ref{fig:iv} and~\ref{fig:iv_cyclic}.
    \end{itemize}
\end{theorem}

\textit{Proof sketch.}\quad For each graph, we use the Lyapunov equation (Eq.~\eqref{eq:lyapunov}) to obtain a system of equations in the unknown drift and diffusion parameters. Imposing $D\in PDD_d$ and $\Sigma\in F_G$, for the graphs~\ref{fig:hy},~\ref{fig:chain},~\ref{fig:iv} and~\ref{fig:iv_cyclic}, we obtain an explicit dependence of $\operatorname{sign}(\alpha)$ on entries of $\Sigma$, which yields sign identifiability. For the remaining graphs, we set $\alpha=0$ and determine for which covariance matrices $\Sigma$ this leads to a contradiction. Such $\Sigma$'s cannot lie in $\mathcal{M}^{0}_{G,\alpha}$. By the $\mathcal{M}^{0}$-criterion (Theorem~\ref{thm:m0-criterion}), this separates covariance matrices for which the sign is identifiable from those for which both signs remain compatible. Applying this analysis yields partial identifiability for both~\ref{fig:confounding} and~\ref{fig:three_loop}. 

\begin{remark}
    For the graph in Fig.~\ref{fig:confounding}, we show in the proof in Appendix~\ref{sec:appendix_proof_no_latent_confoundin}, that partial identifiability holds with positive measure. Showing that partial identifiability constitutes a genuine edge-sign identifiability category.
\end{remark}

\begin{remark}\label{remark:confounding_difference_fixed_difussion}
    The (red) edge $\alpha$ in the graph in Fig.~\ref{fig:confounding} is an example where assuming a fixed diffusion matrix leads to different results. According to Theorem~\ref{thm:no_latent_id}, the (red) edge $\alpha$ is partially identifiable. In contrast, when considering a fixed diffusion matrix $C$, then the same (red) edge $\alpha$ is identifiable \citep[Theorem~5.3]{dettling2023identifiability}.
\end{remark}

The following two lemmas are obtained in the course of the proof of Theorem~\ref{thm:no_latent_id}.

\begin{lemma}[Conditions for Partial Sign Identifiability]\label{lem:no_latent_part_id_values}
    For the red edge $\alpha$ in the graphs shown in Figs.~\ref{fig:confounding} and~\ref{fig:three_loop}, the edge $\alpha$ becomes identifiable under additional conditions on the covariance matrix $\Sigma\in\mathcal{M}^{p}_{G,\alpha}$. 
    These conditions are stated in Appendix~\ref{sec:appendix_covariance_condititions}.
\end{lemma}

\begin{lemma}[Sign Expressions for Sign Identifiable Edges]\label{lem:no_latent_id_values}
    For the sign identifiable red edge $\alpha$ from the graphs shown in Figs.~\ref{fig:hy}, \ref{fig:chain}, \ref{fig:iv}, and~\ref{fig:iv_cyclic}, the sign of $\alpha$ can be expressed in terms of $\Sigma\in\mathcal{M}^p_\alpha$ as follows:
    \begin{itemize}
        \item for~\ref{fig:hy}:  $\operatorname{sign}(\alpha) = \operatorname{sign}(\sigma_{hy})$,
        \item for~\ref{fig:chain}: $\operatorname{sign}(\alpha)=\operatorname{sign}(\sigma_{hy})/\operatorname{sign}(\sigma_{hx})$,
        \item for~\ref{fig:iv}: $\operatorname{sign}(\alpha)=\operatorname{sign}(\sigma_{zy})/\operatorname{sign}(\sigma_{zx})$.
        \item for~\ref{fig:iv_cyclic}: 
        
    $\operatorname{sign}(\alpha)=\begin{cases}
        \operatorname{sign}(\sigma_{zy}/\sigma_{zx})\qquad \text{, if }\, \rho_{zy}\rho_{xy}/\rho_{zx}<1,\\
        -\operatorname{sign}(\sigma_{zy}/\sigma_{zx})\quad \text{, if }\, \rho_{zy}\rho_{xy}/\rho_{zx}>1.
    \end{cases}$
    \end{itemize}
\end{lemma}

\begin{theorem}[Graphical Edge-Sign Identifiability]\label{thm:identifiability}
    For the red edge $\alpha$ in the graphs in Figs.~\ref{fig:hy},\ref{fig:chain},\ref{fig:iv} and~\ref{fig:iv_cyclic}, the sign of $\alpha$ is identifiable.
\end{theorem}

\begin{proof}
    Let the graphs shown in Figs.~\ref{fig:hy},\ref{fig:chain},\ref{fig:iv} and~\ref{fig:iv_cyclic} be the original graphs $G_0,G_1,G_2$ and $G_3$ and let $G'_0,G'_1,G'_2$ and $G'_3$ denote the corresponding graphs obtained by removing the red edge $\alpha$. 
    Comparing $G_0$ with $G'_0$ and $G_1$ with $G'_1$, we see that $H$ and $Y$ no longer share a common ancestor in $G'_0$ and $G'_1$, respectively. 
    Comparing $G_2$ with $G'_2$ and $G_3$ with $G'_3$, we see that $Z$ and $Y$ no longer share a common ancestor in $G'_2$ and $G'_3$, respectively. 
    Hence, the marginal independences implied by $G$ and $G'$ differ in each case, and the graphical criterion (Theorem~\ref{thm:graphical_criterion}) yields identifiability of $\alpha$ for all four graphs.
\end{proof}

\begin{table*}
\centering

\begin{tabular}{lccccccccc}
\toprule
Graph in Fig. &~\ref{fig:hy} &~\ref{fig:chain} &~\ref{fig:confounding} &~\ref{fig:three_loop} &~\ref{fig:iv} &~\ref{fig:iv_cyclic} &~\ref{fig:one_proxy} &~\ref{fig:two_proxy} &~\ref{fig:two_proxy_circulair} ,\\
\midrule
Edge-sign identifiable & 1.0 &1.0 & 0.44 & 0.64 & 1.0 & 1.0 & 0.85 & 1.0 & 1.0 ,\\
\bottomrule
\end{tabular}
\caption{Empirical fraction in $[0,1]$ of sampled covariance matrices $\Sigma \in \mathcal{M}^p_{G,\alpha}$ for which the red edge $\alpha$ is sign identifiable (graphs in Fig.~\ref{fig:graphs}, no latent variables). For a discussion on these numerical results, see Section~\ref{sec:numerical_edge}.}
\label{tab:identifiability}
\end{table*}

\subsubsection{Latent Variables}\label{sec:latent}
\begin{theorem}\label{thm:latent_id}
    For the red edge $\alpha$ in the graphs in Fig.~\ref{fig:hy}, \ref{fig:confounding},~\ref{fig:iv} and~\ref{fig:iv_cyclic} (with $H$ being latent) under the $t$-faithfulness assumption, the sign of $\alpha$ is:
    \begin{itemize}
        \item  non-identifiable for~\ref{fig:hy} and~\ref{fig:confounding},
        \item identifiable for~\ref{fig:iv} and~\ref{fig:iv_cyclic}.
    \end{itemize}
\end{theorem}
\textit{Proof sketch.}\quad We build on the proofs from the no latent case (Theorem~\ref{thm:no_latent_id}). The only change is that $H$ is now latent, so covariance entries involving $H$ are unobserved. We therefore treat the corresponding blocks (e.g., $\Sigma_{ho},\Sigma_{oh}, \Sigma_{hh}$) as free variables that can vary subject to $\Sigma\in \mathcal{M}^p_{G,\alpha}$. This additional freedom allows us to choose latent-dependent covariance entries so that some scenarios that were (partially) identifiable in the fully observed case become non-identifiable. Note that ensuring $\Sigma\in \mathcal{M}^p_{G,e}$ is one of the main challenges in this proof. 

\begin{remark}
    In particular, among the sign expressions in Lemma~\ref{lem:no_latent_id_values}, only for the (cycle with) instrumental variable case (Fig.~\ref{fig:iv} and~\ref{fig:iv_cyclic}), we still have sign-identifiability in the latent variable setting.
\end{remark}

\section{Numerical Results}\label{sec:numerical_results}
This section reports numerical results on sign-identifiability of the red edge $\alpha$ for the graphs in Fig.~\ref{fig:graphs} (no latent variables). The results are summarized in Table~\ref{tab:identifiability}. Each entry is the empirical fraction in $[0,1]$ of identifiable instances of $\alpha$ across 1000 independently generated samples produced by the algorithm described below. If the fraction equals $1$ (resp.\ $0$), $\alpha$ is identifiable (resp.\ non-identifiable). If the fraction lies in $(0,1)$, $\alpha$ is partially identifiable.\footnote{For the implementation, see the following link: \href{https://github.com/ggvanseev/UAI_Sign-Identifiability-of-Causal-Effects-in-Stationary-Stochastic-Dynamical-Systems}{repository}. The repository includes a guiding example to test any desired custom causal structure.}

\subsection{Method}\label{sec:numerical_results_method}
Throughout this section, we repeatedly use that for $D\in PDD_d\subset PD_d$ in the Lyapunov equation (Eq.~\eqref{eq:lyapunov}),\, $\Sigma\in PD_d$ if and only if $A$ is Hurwitz stable \citep{frommer2012verified}.

For a fixed graph $G=(V,E)$, samples are generated as follows. We define a symbolic drift matrix $A_{sym}$, diffusion matrix $D_{sym}$ and covariance matrix $\Sigma$, where $\operatorname{supp}(A_{sym})=E$, $D_{sym}$ is diagonal and $\Sigma_{sym}$ respects the marginal independences of the graph.
We first draw a Hurwitz stable drift matrix $A$ with $\operatorname{supp}(A)=E$ and a diagonal matrix $D \in PDD_d$. Each non-zero entry from $(A_{sym}, D_{sym})$ is sampled uniformly from a bounded interval (e.g., $A_{ij} \sim U(-10,10)$), restricting to the appropriate domain when the sign is known (e.g., $D_{ii} \sim U(0,10)$). Since $(A,D)$ can be rescaled by any $a \in \mathbb{R}_+$ without changing $\Sigma$, this sampling effectively explores the parameter space. We resample $A$ until it is Hurwitz stable. Given $(A,D)$, we solve the Lyapunov equation to obtain $\Sigma\in PD_d$. We then verify if $\Sigma$ respects the marginal independences by comparison with the zero pattern in $\Sigma_{sym}$, if this fails we reject the sample. If it passes, $\Sigma\in\bigcup_{e \in E}\mathcal{M}^p_{G,e}$.

For a fixed edge $e \in E$, we then test sign-identifiability by searching for $(A',D')$ with $D' \in PDD_d$ and $\operatorname{sign}(A'_e) = -\operatorname{sign}(A_e)$ that satisfies the Lyapunov equation with $\Sigma$. This feasibility problem is solved numerically (see remark below). Since $\Sigma \in PD_d$ and $D' \in PDD_d$, any feasible $A'$ is necessarily Hurwitz stable. If such $(A',D')$ exists, $e$ is declared non-identifiable; otherwise it is identifiable (Definition~\ref{def:pointwise_edge_sign_identifiability}).This procedure is repeated for 1000 independent samples. See Appendix~\ref{sec:appendix_algo} for the pseudocode.

\begin{remark}\label{remark:feasability_problem}
    Without latent variables, the Lyapunov equation induces a linear system in the unknowns $(A',D')$ for fixed $\Sigma$. Hence, testing feasibility reduces to a linear optimisation (or feasibility) problem, which admits sound and complete polynomial-time algorithms.\footnote{To see how the graph size in terms of the number of nodes scales, see Appendix~\ref{app:numerical_benchmark} for a numerical benchmark.} Consequently, if a Hurwitz stable matrix $A$ can be sampled in polynomial time, the overall procedure runs in polynomial time.
    \\
    In contrast, with latent variables, the Lyapunov constraints become bilinear: the unobserved covariance entries $\sigma_{hh} \in \Sigma$ interact multiplicatively with drift coefficients $A_{ij}$. The resulting feasibility problem is bilinear and therefore NP-hard \citep{JMLR:v12:petrik11a}. For this reason, we restrict our experiments to graphs without latent variables.
\end{remark}

\subsection{Edge-Sign Identifiability}\label{sec:numerical_edge}
Table~\ref{tab:identifiability} reports the sign identifiability results for the red edge $\alpha$ in Fig.~\ref{fig:graphs}. The second row shows the empirical fraction of samples in which the edge is identifiable.
The numerical results yield three observations. First, the empirical fractions are fully consistent with Theorem~\ref{thm:no_latent_id}: sign-identifiable graphs have fraction $1$, non-identifiable graphs have fraction $0$, and partially identifiable graphs yield fractions in $(0,1)$. Second, for the graphs in Figs.~\ref{fig:one_proxy}–\ref{fig:two_proxy_circulair} (not analysed in Section~\ref{sec:no_latent}), the sign appears identifiable in Figs.~\ref{fig:two_proxy} and~\ref{fig:two_proxy_circulair}, and partially identifiable in Fig.~\ref{fig:one_proxy}. Third, in the partially identifiable category, both outcomes occur with substantial frequency (identifiable fractions $0.44$, $0.64$, and $0.85$ for Figs.~\ref{fig:confounding},~\ref{fig:three_loop}, and~\ref{fig:one_proxy}), indicating the need to verify sign identifiability for the specific covariance matrix under consideration for those structures. 

\section{Conclusion}
We studied identifiability in continuous linear stationary SDEs given the causal graph, where we relaxed the assumption of knowing the diffusion matrix $C$. In this setup, the linear SDE is scale invariant (when $C$ is not fixed), therefore, we aimed to identify the sign of a given edge. We introduced edge-sign identifiability and characterised three notions of sign identifiability, namely, identifiable, non-identifiable, and partially identifiable. Of which, partial identifiability forms a new category in the literature. Using these categories, we derived general criteria characterising when the sign of an edge can be determined from the observational covariance matrix $\Sigma$, given the causal graph.  We illustrated the applicability of our results on classical structures, including an instrumental variable setting, for which we obtained an explicit sign in terms of $\Sigma$. Moreover, we showed that in the confounding setting, partial identifiability has positive measure, thereby showing that partial identifiability is a genuine category. Numerical experiments further indicate that when an edge is partially identifiable, both identifiable and non-identifiable constitute substantial sub-categories. Future directions include extensions to subgraphs and graph-level sign identifiability.

\bibliography{uai2026-template}

\newpage

\onecolumn

\title{Sign Identifiability of Causal Effects in Stationary Stochastic Dynamical Systems,\\(Supplementary Material)}
\maketitle

\appendix

\section{Clarification Pointwise Edge-Sign Identifiability}\label{sec:appendix_pointwise_clarification}
For a fixed $\Sigma \in \mathcal{M}^p_{G,e}$, Definition~\ref{def:edge_identifiability} amounts to intersecting the corresponding sets with the singleton $\{\Sigma\}$. Hence,
\begin{itemize}
        \item non-identifiable if $\mathcal{M}^{+}_{G,e}\cap\{\Sigma\}=\{\Sigma\}=\mathcal{M}^{-}_{G,e}\cap\{\Sigma\}$ while $\mathcal{M}^{-}_{G,e}\cap\{\Sigma\}\neq\emptyset$,
        \item partially identifiable if $\mathcal{M}^{+}_{G,e}\cap\{\Sigma\}\neq\mathcal{M}^{-}_{G,e}\cap\{\Sigma\}$ and $\mathcal{M}^{+}_{G,e}\cap\mathcal{M}^{-}_{G,e}\cap\{\Sigma\}\neq\emptyset$,
        \item identifiable if $\mathcal{M}^{+}_{G,e}\cap\mathcal{M}^{-}_{G,e}\cap\{\Sigma\}=\emptyset$ while $\mathcal{M}^{-}_{G,e}\cap\{\Sigma\}\neq\emptyset$ or $\mathcal{M}^{+}_{G,e}\cap\{\Sigma\}\neq\emptyset$.
    \end{itemize}
Therefore, non-identifiability reduces to $\Sigma\in\mathcal{M}^{+}_{G,e}\cap\mathcal{M}^{-}_{G,e}$ and identifiability reduces to $\Sigma\not\in\mathcal{M}^{+}_{G,e}\cap\mathcal{M}^{-}_{G,e}$. Since $\mathcal{M}^{+}_{G,e}\cap\{\Sigma\}\neq\mathcal{M}^{-}_{G,e}\cap\{\Sigma\}$ implies that $\Sigma\not\in\mathcal{M}^{+}_{G,e}\cap\mathcal{M}^{-}_{G,e}$, whereas $\mathcal{M}^{+}_{G,e}\cap\mathcal{M}^{-}_{G,e}\cap\{\Sigma\}\neq\emptyset$ implies that $\Sigma\not\in\mathcal{M}^{+}_{G,e}\cap\mathcal{M}^{-}_{G,e}$, we obtain a contradiction. Hence, partial identifiability cannot be meaningfully defined for a fixed $\Sigma$. Summarising this gives us the pointwise edge-sign identifiability Definition~\ref{def:pointwise_edge_sign_identifiability}.

\section{Covariance Conditions Lemma~\ref{lem:no_latent_part_id_values}}\label{sec:appendix_covariance_condititions}
For the sign identifiable edges $\alpha$ from the graphs $G$ shown in Fig.~\ref{fig:confounding} and~\ref{fig:three_loop} with the matching covariance matrices $\Sigma\in\mathcal{M}^p_\alpha$, $\alpha$ is identifiable when
    \begin{itemize}
        \item for~\ref{fig:confounding} if one of the following conditions holds,
        \begin{equation}
            \begin{split}
            (c.1)&\quad  \frac{(2\rho_{hy}^2\rho_{hx}^2-\rho_{hy}^2-\rho_{hx}^2)(\rho_{hx}\rho_{hy}-\rho_{xy})}{2\rho_{hx}\rho_{hy}(\rho_{hx}^2-1)(\rho_{hy}^2-1)}\geq1,\\
            (c.2)&\quad\,\,\operatorname{sign}(\sigma_{hx}\sigma_{hy})\neq\operatorname{sign}(\sigma_{xy}),
            \\
    (c.3)&\quad \frac{\rho_{hy}\rho_{hx}}{\rho_{xy}}\geq1.
                    \end{split}
        \end{equation}
        
        \item for~\ref{fig:three_loop} if one of the following conditions hold
        \begin{equation}
            \begin{split}
                (c.1)&\quad \text{If $d>0,\,a<0$ and $b<0$, and $(-a+c)/b\leq1$}, \\
                (c.2)&\quad  \text{If $d>0,\,a>0$ and $b>0$, and $(-a+c)/b\leq1$},\\
                (c.3)&\quad  \text{If $d<0,\,a<0$ and $b>0$, and $(-a+c)/b\geq1$},\\
                (c.4)&\quad \text{If $d<0,\,a>0$ and $b<0$, and $(-a+c)/b\geq1$},\\
                (c.5)&\quad d=1,\\
            \end{split}
        \end{equation}
        where
        \begin{equation}
            \begin{split}
            a&:= \frac{\rho_{xy}^2}{1-\rho_{xy}^2}\left(\rho_{hx}-\rho_{hy}\rho_{xy}\right),\\
            b&:=\frac{\rho_{xy}\rho_{hx}}{\rho_{xy}-\rho_{hy}\rho_{hx}}\left(\rho_{hx}-\frac{\rho_{hy}}{\rho_{xy}}\right),\\
            c&:=\frac{\rho_{hy}}{\rho_{xy}}+\rho_{xy}\rho_{hy},\\
            d&:= \frac{\rho_{hx}\rho_{hy}}{\rho_{xy}}.
                    \end{split}
        \end{equation}
    \end{itemize}

\section{Algorithm for Numerical Experiments}\label{sec:appendix_algo}

\begin{algorithm}
\caption{Determining Sign Identifiability}
\begin{algorithmic}[1]
\Require Symbolic: drift matrix $A_{sym}$, diffusion matrix $D_{sym}$, ($t$-faithful) covariance matrix $\Sigma_{sym}$ and edge e
\State $N_{samples}\gets1000$
\State Identifiable$\gets 0$
\For{$i \gets 1$ to $N_{\text{samples}}$}
    \State $A\in\text{Hurwitz stable},D\in PDD\gets \operatorname{draw Param}(A_{sym},D_{sym})$
    \State $\Sigma\in PD \gets \operatorname{calcSigma} (A,D,\Sigma_{sym})$
    \If{not $\Sigma\in F_G$}
    \State $N_{samples}-=1$
    \ElsIf{ not  $\operatorname{oppSol}(\text{e}, A,\Sigma, A_{sym},D_{sym})$}
        \State Identifiable $+=1$
    \EndIf 
\EndFor
\State Fraction $\gets\text{Identifiable}/N_{samples}$
\State \Return $-1$
\end{algorithmic}
\label{algo:sign-id}
\end{algorithm}

\subsection{Numerical Benchmark of Runtime Finding an Opposite Solution}\label{app:numerical_benchmark}
To indicate the scalability of linear programming (LP) for determining the feasibility of a sign of a desired edge for a given causal structure $G$ and covariance matrix $\Sigma$ we ran an experiment. Note that this corresponds to the function $\operatorname{oppSol}(\text{e}, A,\Sigma, A_{sym},D_{sym})$ in Algorithm~\ref{algo:sign-id}. In this experiment, we determined the runtime as a function of the number of nodes for different average degrees. For each number of nodes $n$ and average degree $d$, we drew a sample of 100 random Erdős-Rényi graphs $G(n,m)$ with $m=nd$ edges. Per random graph in the sample, we tracked the wall-clock runtime of checking the feasibility of a sign (for a given causal structure $G$ and covariance matrix $\Sigma$). For each sample, we calculated the mean and the standard deviation in the runtime. The results of the experiment are shown in Table~\ref{tab:runtime}.

\begin{table}[ht]
\centering
\caption{Runtime as a function of the number of nodes for different average degrees. Values are wall-clock runtime in seconds, reported as mean $\pm$ standard deviation.}
\label{tab:runtime-nodes-avgdeg}
\begin{tabular}{rrrr}
\toprule
Nodes & Avg. degree $2$ & Avg. degree $3$ & Avg. degree $4$ \\
\midrule
10    & $(2.41 \pm 1.74)\times 10^{-3}$ & $(2.55 \pm 0.183)\times 10^{-3}$ & $(3.12 \pm 0.226)\times 10^{-3}$ \\
25    & $(5.52 \pm 0.569)\times 10^{-3}$ & $(7.06 \pm 0.740)\times 10^{-3}$ & $(9.00 \pm 0.683)\times 10^{-3}$ \\
50    & $(1.41 \pm 0.264)\times 10^{-2}$ & $(2.16 \pm 0.323)\times 10^{-2}$ & $(2.91 \pm 0.397)\times 10^{-2}$ \\
100   & $(4.20 \pm 0.0496)\times 10^{-2}$ & $(7.43 \pm 1.76)\times 10^{-2}$ & $(1.09 \pm 0.176)\times 10^{-1}$ \\
250   & $(2.68 \pm 0.0463)\times 10^{-1}$ & $(3.14 \pm 0.757)\times 10^{-1}$ & $(7.25 \pm 2.69)\times 10^{-1}$ \\
500   & $1.18 \pm 0.0388$ & $1.17 \pm 0.00915$ & $2.10 \pm 1.18$ \\
1,000 & $5.34 \pm 0.248$ & $4.86 \pm 0.0989$ & $5.25 \pm 0.190$ \\
2,500 & $37.3 \pm 2.37$ & $34.2 \pm 3.47$ & $33.8 \pm 0.746$ \\
5,000 & $159 \pm 14.9$ & $162 \pm 20.5$ & $144 \pm 22.5$ \\
\bottomrule
\end{tabular}
\label{tab:runtime}
\end{table}

\section{Proofs}
\subsection{Lemma~\ref{lem:m0_implications}}\label{sec:appendix_lemma_m0_implications}
\begin{proof}
    Throughout the proof, let \(G=(V,E)\) be a graph and \(e \in E\) a fixed edge.
    To start, let  $\Sigma\in \mathcal{M}^{+}_e\cap\mathcal{M}^{-}_e$. Then we obtain the following two solutions to the Lyapunov equation (Eq.~\eqref{eq:lyapunov})
\begin{equation}
    \begin{split}
    aA^+\Sigma+\Sigma a{A^+}^T &=-aD^+,\\
     bA^-\Sigma+\Sigma b{A^-}^T &=- bD^-,
    \end{split}
\end{equation}
with $a,b\in\mathbb{R}$, $A^k$ a Hurwitz stable matrix, and $D^k\in PDD_d$ where $k\in\{+,-\}$. The sum of the two solutions is again a valid equation. Hence,
\begin{equation}
    \begin{split}
    aA^+\Sigma+\Sigma a{A^+}^T + bA^-\Sigma+\Sigma b{A^-}^T&=-aD^+ -bD^- ,\\
    (aA^++bA^-)\Sigma+\Sigma (aA^++bA^-)^T &=-aD^+ -bD^- ,\\
    A\Sigma+\Sigma A^T&=-D,
    \end{split}
\end{equation}
where $A=aA^++bA^-$ and $D=aD^++bD^-$. For any $x\in \mathbb{R}^d$ where $x\neq 0$, with $a=t$ and $b=1-t$, $t\in[0,1]$, we have:
\begin{equation}
    \begin{split}
    x^T(tD^++(1-t)D^-)x&>0 ,\\
     tx^TD^+x+(1-t)x^TD^-x&>0.
    \end{split}
\end{equation}
Since $D^+,D^-\in PDD_d$ this is true for any choice of $t\in[0,1]$. Therefore, $D\in PDD_d$. In addition, due to $A_{ij}=0$ for any $(j,i)\not\in E$, $\operatorname{supp}(tA^++(1-t)A^-)\subseteq E$. Furthermore, we can pick some $t\in (0,1)$ such that $tA^+_e + (1-t)A^-_e=0$. Then, $\operatorname{sign}(A_e)=\operatorname{sign}(tA^+_e + (1-t)A^-_e)=0$. Finally, $\Sigma\in \mathcal{M}^{+}_e\cap\mathcal{M}^{-}_e\subseteq F_G$. Therefore, according to Definition~\ref{def:edge_signature_set}, we have $\Sigma\in\mathcal{M}^{0}_e$. To summarise, 
\begin{equation}
    \Sigma\in \mathcal{M}^{+}_e \text{ and }\Sigma\in\mathcal{M}^{-}_e \implies \Sigma\in \mathcal{M}^{0}_e.
\end{equation}

Furthermore, let $\Sigma\in \mathcal{M}^{+}_e\cap\mathcal{M}^{0}_e$. Then, we obtain the following two solutions to the Lyapunov equation (Eq.~\eqref{eq:lyapunov})
\begin{equation}
    \begin{split}
    aA^+\Sigma+\Sigma a{A^+}^T &= -aD^+ ,\\
     bA^0\Sigma+\Sigma b{A^0}^T &= -bD^0,
    \end{split}
\end{equation}
with $a,b\in\mathbb{R}$, $A^k$ a Hurwitz stable matrix and $D^K\in PD_d$ where $k\in\{+,0\}$. The sum of the two solutions is again a valid equation. Hence,
\begin{equation}
    \begin{split}
    aA^+\Sigma+\Sigma a{A^+}^T + bA^0\Sigma+\Sigma b{A^0}^T&=-aD^+-bD^0 ,\\
    (aA^++bA^0)\Sigma+\Sigma (aA^++bA^0)^T &=-aD^+-bD^0 ,\\
    A\Sigma+\Sigma A^T&=-D,
    \end{split}
\end{equation}
where $A=aA^++bA^0$ and $D=aD^++bD^0$. If we pick $a=-1$ to get
\begin{equation}
    \begin{split}
    x^T(-D^++bD^0)x&>0 ,\\
    x_i(-D^+_{ii}+bD_{ii}^0)x_i&\overset{(a)}{>}0 ,\\
    (-D^+_{ii}+bD_{ii}^0)x_i^2&>0 ,\\
    (-D^+_{ii}+bD_{ii}^0)&>0 ,\\
    bD_{ii}^0&>D^+_{ii} ,\\
    b&>\frac{D^+_{ii}}{D_{ii}^0},
    \end{split}
\end{equation}
where $(a)$ use that $D$ is diagonal. If $b>\operatorname{max}_{i}\big(D^+_{ii}/D^0_{ii}\big)$, then $D\in PDD_d$. Since $b$ is unbounded, we can pick $b$ such that $D\in PDD_d$. Moreover, since $A_{ij}=-A^+_{ij} + bA^0_{ij}$ and $b$ is still unbounded from above, we can choose $b$  sufficiently large such that $\operatorname{supp}(-A^++bA^0)= E$. Furthermore, $\operatorname{sign}(A_e)=\operatorname{sign}(-A^+_e + bA^0_e)=\operatorname{sign}(-A^+_e)=-\operatorname{sign}(A^+_e)=-$ for any choice of b. Finally, $\Sigma\in \mathcal{M}^{+}_e\cap\mathcal{M}^{0}_e\subseteq F_G$. Therefore, according to Definition~\ref{def:edge_signature_set}, if we pick $b$  such that $D\in PDD_d$ and $\operatorname{supp}(-A^++bA^0)= E$, then $\Sigma\in \mathcal{M}^{-}_e$. To summarise
\begin{equation}
    \Sigma\in \mathcal{M}^{+}_e\text{ and }\Sigma\in\mathcal{M}^{0}_e \implies \Sigma\in \mathcal{M}^{-}_e .
\end{equation}
Finally, we can analogously show, 
\begin{equation}
   \Sigma\in \mathcal{M}^{-}_e\text{ and }\Sigma\in\mathcal{M}^{0}_e \implies \Sigma\in \mathcal{M}^{+}_e .
\end{equation}
\end{proof}

\subsection{Lemma~\ref{lem:extending_m0}}\label{sec:appendix_extending_m0}
\begin{proof}
        Throughout the proof, let \(G=(V,E)\) be a graph and \(e \in E\) a fixed edge. We define the edge signature sets  for $D\in PD_d$ as follows:
\begin{definition}[PD Edge Signature Set]\label{def:pd_edge_signature_set}
    For a graph $G=(V,E)$ and edge $e\in E$, we define the edge signature set $\mathcal{M}^{k}_{G,e}$ as
    \begin{equation}
        \begin{split}
        \mathcal{M}^{k}_{G,e}:=&\{\Sigma\in PD_d\,: \exists A,D \text{ s.t. }A\Sigma + \Sigma A^T= -D;\, \\
        &D\in PD_d; \text{supp}(A)= E;\, \text{sign}(A_e)=k\},
        \end{split}
\end{equation}
    where $k\in\{+,-\}$ and
\begin{equation}
    \begin{split}
        \mathcal{M}^{0}_{G,e}:=&\{\Sigma\in PD_d\,: \exists A,D \text{ s.t. }A\Sigma + \Sigma A^T= -D;\, \\
        &D\in PD_d; \text{supp}(A)\subset E;\, \text{sign}(A_e)=0\}.
    \end{split}
\end{equation}
\end{definition}
Since we do not know of any known constraints on the (marginal) independences in the literature for a graph $G$ with $D\in PD$, we no longer enforce \textit{$t$-faithfulness} $F_G$.

    To start, let  $\Sigma\in \mathcal{M}^{+}_e\cap\mathcal{M}^{-}_e$.Then, we obtain the following two solutions to the Lyapunov equation (Eq.~\eqref{eq:lyapunov}):
\begin{equation}
    \begin{split}
    aA^+\Sigma+\Sigma a{A^+}^T &=-aD^+ ,\\
     bA^-\Sigma+\Sigma b{A^-}^T &= -bD^-,
    \end{split}
\end{equation}
with $a,b\in\mathbb{R}$, $A^k$ a Hurwitz stable matrix, and $D^k\in PD_d$ where $k\in\{+,-\}$. The sum of the two solutions is again a valid equation. Hence,
\begin{equation}
    \begin{split}
    aA^+\Sigma+\Sigma a{A^+}^T + bA^-\Sigma+\Sigma b{A^-}^T&=-aD^+-bD^- ,\\
    (aA^++bA^-)\Sigma+\Sigma (aA^++bA^-)^T &=-aD^+-bD^- ,\\
    A\Sigma+\Sigma A^T&=-D,
    \end{split}
\end{equation}
where $A=aA^++bA^-$ and $D=aD^++bD^-$. We pick $a=t$ and $b=1-t$ with $t\in[0,1]$. For any $x\in \mathbb{R}^d$ where $x\neq 0$, we have that
\begin{equation}
    \begin{split}
    x^T(tD^++(1-t)D^-)x&>0 ,\\
     tx^TD^+x+(1-t)x^TD^-x&>0.
    \end{split}
\end{equation}
Since $D^+,D^-\in PD_d$ this is true for any choice of $t\in[0,1]$, therefore $D\in PD_d$. In addition, due to $A_{ij}=0$ for any $(j,i)\not\in E$, 
$\operatorname{supp}(tA^++(1-t)A^-)\subseteq E$. Furthermore, we can pick $t$ such that $tA^+_e + (1-t)A^-_e=0$. Then, $\operatorname{sign}(A_e)=\operatorname{sign}(tA^+_e + (1-t)A^-_e)=0$. In addition, $\Sigma\in \mathcal{M}^{+}_e\cap\mathcal{M}^{-}_e\subseteq PD_d$. Therefore, according to Definition~\ref{def:pd_edge_signature_set},
we have: $\Sigma\in\mathcal{M}^{0}_e$. To summarise, 
\begin{equation}
    \begin{split}
    \Sigma\in \mathcal{M}^{+}_e \text{ and }\Sigma\in\mathcal{M}^{-}_e \implies \Sigma\in \mathcal{M}^{0}_e.
    \end{split}
\end{equation}

Furthermore, let $\Sigma\in \mathcal{M}^{+}_e\cap\mathcal{M}^{0}_e$. Then, we obtain the following two solutions to the Lyapunov equation (Eq.~\eqref{eq:lyapunov})
\begin{equation}
    \begin{split}
    aA^+\Sigma+\Sigma a{A^+}^T &=-aD^+ ,\\
     bA^0\Sigma+\Sigma b{A^0}^T &= -bD^0,
    \end{split}
\end{equation}
with $a,b\in\mathbb{R}$, $A^k$ a Hurwitz stable matrix and $D^K\in PD_d$ where $k\in\{+,0\}$. The sum of the two solutions is again a valid equation. Hence,
\begin{equation}     \begin{split}
    aA^+\Sigma+\Sigma a{A^+}^T + bA^0\Sigma+\Sigma b{A^0}^T&=-aD^+-bD^0 ,\\
    (aA^++bA^0)\Sigma+\Sigma (aA^++bA^0)^T &=-aD^+-bD^0 ,\\
    A\Sigma+\Sigma A^T&=-D,
\end{split} \end{equation}
where $A=aA^++bA^0$ and $D=aD^++bD^0$. 
Let $a=-1$. For any $x\in \mathbb{R}^d$ where $x\neq 0$, we have that 
\begin{equation}     \begin{split}
    x^T(-D^++bD^0)x&>0 ,\\
    bx^TD^0x&>x^TD^+x ,\\
    b&>\frac{x^TD^+x}{x^TD^0x}  ,\\
     b&>\frac{x^TD^+x}{x^TD^0x}\frac{x^Tx}{x^Tx} ,\\
     b&\overset{(a)}{>}\frac{R(D^+,x)}{R(D^0,x)} ,\\
     b&\overset{(b)}{>}\frac{\lambda_{max}^+}{\lambda_{min}^0},
\end{split} \end{equation}
where $(a)$ use the definition of the Rayleigh quotient $R$, $(b)$ use that $\lambda_{max}^+/\lambda_{min}^0\geq R(D^+,x)/R(D^0,x)$, and $\lambda^k$ are the eigenvalues of $D^k$. Since $D^k\in PD_d$ we have that $\lambda_{min}^k>0$. In addition, as $b$ is unbounded, we can pick a value of $b$ that satisfies the inequality. Therefore, we can always pick $b$ such that $D\in PD_d$. Moreover, since $A_{ij}=-A^+_{ij} + bA^0_{ij}$ and $b$ is still unbounded from above, we can also pick $b$ sufficiently large to get $\operatorname{supp}(-A^++bA^0)= E$. Furthermore, $\operatorname{sign}(A_e)=\operatorname{sign}(-A^+_e + bA^0_e)=\operatorname{sign}(-A^+_e)=-\operatorname{sign}(A^+_e)=-$ for any choice of b. Finally, $\Sigma\in \mathcal{M}^{+}_e\cap\mathcal{M}^{0}_e\subseteq PD_d$. Therefore, according to Definition~\ref{def:pd_edge_signature_set}, if we pick $b$  such that $D\in PD_d$ and $\operatorname{supp}(-A^++bA^0)= E$, then $\Sigma\in \mathcal{M}^{-}_e$. To summarise
\begin{equation}
    \Sigma\in \mathcal{M}^{+}_e\text{ and }\Sigma\in\mathcal{M}^{0}_e \implies \Sigma\in \mathcal{M}^{-}_e .
\end{equation}
Moreover, we can analogously show, 
\begin{equation}
   \Sigma\in \mathcal{M}^{-}_e\text{ and }\Sigma\in\mathcal{M}^{0}_e \implies \Sigma\in \mathcal{M}^{+}_e .
\end{equation}

We have now proven that Lemma~\ref{lem:m0_implications} can be extended to $D\in PD_d$. 

The proof for the $\mathcal{M}^0$ criterion Theorem~\ref{thm:m0-criterion} follows from the same arguments when using the adjusted Definition~\ref{def:pd_edge_signature_set} and extend form of Lemma~\ref{lem:m0_implications} in the original proof (see below Theorem~\ref{thm:m0-criterion}).
\end{proof}

\subsection{Lemma~\ref{lem:edge_deletion_Fg}}\label{app:proof_edge_deletion_Fg}

\begin{proof}
Let $G=(V,E)$, $G'=(V,E')$ and $G''=(V,E'')$ with
$E''\subset E'\subset E$, i.e. $G'$ is obtained by deleting edges from $G$, and $G''$ by deleting edges from $G'$. Denote $MI_H$ as the set of pairs of marginally independent nodes of a graph $H$. The set of pairs of marginally independent
nodes $MI_H$ can only increase or stay the same due to the deletion of edges, since this deletion can only remove ancestors and not add new ones. Therefore, $MI_{G}\subseteq MI_{G'}\subseteq MI_{G''}$. In addition, note that for two graphs $H$ and $H'$ to have the same $t$-faithful set $F_H$, they need to have $MI_H=MI_{H'}$. Therefore, if $F_G \neq F_{G'}$, then $MI_{G}\subset MI_{G'}\subseteq MI_{G''}$. Hence, $MI_{G}\subset MI_{G''}$ and we get the desired result $F_G \neq F_{G''}$.
\end{proof} 

\subsection{Theorem~\ref{thm:no_latent_id}}\label{sec:proof_no_latent}
In the proofs, we solve the equations that result from comparing the matrix entries from the left and right hand side in the Lyapunov equation (Eq.~\eqref{eq:lyapunov}). We use that matrices on the left and right hand side are symmetric, such that a $d\times d$ matrix results in $a=d(d-1)/2$ equations. To facilitate reading in and comparisons between the proofs, we use the convention of giving the sets of equations consistent numbers $(i)$ to the Roman number of equations $a $ in each proof.

Furthermore, we use the property of triangular matrices that their eigenvalues are on the diagonal. For a triangular Hurwitz stable drift matrix $A\in\mathbb{R}^{d\times d}$, this therefore means that the diagonals all have to be negative. 

In addition, we use that any matrix $B\in PD_d$ has positive diagonal entries, i.e., $B_{ii}>0$. Therefore the diagonal matrix has strict positive diagonals, i.e., $D=D_{ii}>0$. In addition, we will always assume that the covariance matrix $\Sigma\in\mathcal{M}^p_{G,\alpha}$ such that $\Sigma$ is $t$-faithful. Therefore $\Sigma_{ii}>0$. Another property of covariance matrices without exact linear dependencies between random variables $X_i$ and $X_j$, as is the case in our OU process Eq.~\eqref{eq:OU}, is that $|\Sigma_{ij}|<\sqrt{\Sigma_{ii}\Sigma_{jj}}$. For the off-diagonals in the covariance matrix $\Sigma$ we will therefore use the notation that $\Sigma_{ij}=\rho_{ij}\sqrt{\Sigma_{ii}\Sigma_{jj}}$ with $\rho_{ij}\in(-1,1)$ being the correlation coefficient. 

Finally, we can write the covariance matrix as $\Sigma = K R K$, where $K$ is diagonal with  $K_{ii}=\sqrt{\sigma_{ii}}>0$ such that $K\in PDD_d$, and
\begin{equation}
    R_{ij} =
\begin{cases}
1 & \text{if } i=j ,\\
\rho_{ij} & \text{if } i\neq j.
\end{cases}
\end{equation} This means that if and only if $R\in PD_d$, then $\Sigma\in PD_d$. Therefore, if $\Sigma\in PD_3$ we have
\begin{equation}
    R =
\begin{pmatrix}
1 & \rho_{12} & \rho_{13} ,\\
\rho_{12} & 1 & \rho_{23} ,\\
\rho_{13} & \rho_{23} & 1
\end{pmatrix},
\end{equation}
and, by Sylvester's criterion \citep{Horn2012-xj}, $R\in PD_d$ if and only if $1-\rho_{12}^2>0$ and $1+2\rho_{12}\rho_{13}\rho_{23}-\big(\rho_{12}^2+\rho_{13}^2+\rho_{23}^2\big)>0$. Since $|\rho_{ij}|<1$, the first condition is always satisfied, we only need to ensure that  
\begin{equation}\label{eq:sylvester_criterion_3__withtout_latent}
    1+2\rho_{12}\rho_{13}\rho_{23}-\big(\rho_{12}^2+\rho_{13}^2+\rho_{23}^2\big)>0, 
\end{equation}
to show that $R\in PD_3$ and thus $\Sigma\in PD_3$.

\subsubsection{Cause and Effect}\label{sec:proof_no_laten_cause_n_effect}
\begin{proof}
 We adopt the assumptions and conventions stated at the start of this section. Let $G=(V,E)$ be the graph of Fig.~\ref{fig:hy}. The nodes $V=\{H,Y\}$ correspond to the stationary SDE process $X=(H,Y)^T$, then the Hurwitz stable drift matrix $A$ respecting the causal structure of graph $G$ is
\begin{equation}
    A=\left[\begin{matrix}s_{h} & 0,\\\alpha & s_{y}\end{matrix}\right],
\end{equation}
the diagonal diffusion matrix is
\begin{equation}
    D=\left[\begin{matrix}d_{h} & 0,\\0 & d_{y}\end{matrix}\right]\in PDD_2,
\end{equation}
and the covariance matrix is 
\begin{equation}
    \Sigma=\left[\begin{matrix}\sigma_{hh} & \sigma_{hy},\\\sigma_{hy} & \sigma_{yy}\end{matrix}\right]\in\mathcal{M}^p_{G,\alpha}.
\end{equation}

The resulting set of equations to solve is
\begin{equation}     \begin{split}
(i)\;& - d_{h}=2 s_{h} \sigma_{hh}  ,\\
(ii)\;& 0=\alpha \sigma_{hh} + s_{h} \sigma_{hy} + s_{y} \sigma_{hy}  ,\\
(iii)\;& -d_{y}=2\alpha \sigma_{hy} + 2 s_{y} \sigma_{yy}.
\end{split} \end{equation}
Eq.~(ii) is satisfied if and only if $\alpha = b_1 \sigma_{hy}$, where $b_1 = -\frac{s_y + s_h}{\sigma_{hh}}$. Since $s_y,s_h<0$ and $\sigma_{hh}>0$, we have that $b_1>0$ and thus $\operatorname{sign}(\alpha) = \operatorname{sign}(b_1\sigma_{hy}) = \operatorname{sign}(\sigma_{hy})$.  Since $\Sigma\in \mathcal{M}^p_{G,\alpha}$, we have $\sigma_{hy}\neq0$ meaning that the sign of $\alpha$ is $+$ or $-$. Therefore there exists no $\Sigma\in\mathcal{M}^0_{G,\alpha}$, such that by virtue of the $\mathcal{M}^0$ criterion Theorem~\ref{thm:m0-criterion}, for any $\Sigma\in \mathcal{M}^p_{G,\alpha}$, 
the sign of edge $\alpha$ in graph $G$ is identifiable.

\end{proof}

\subsubsection{Chain}\label{sec:appendix_proof_no_latent_chain}

\begin{proof}
     We adopt the assumptions and conventions stated at the start of this section. Let $G=(V,E)$ be the graph of Fig.~\ref{fig:chain}. The nodes $V=\{H,X,Y\}$ correspond to the SDE process $X=(H,X,Y)^T$, then the Hurwitz stable drift matrix $A$ respecting the causal structure of graph $G$ is
\begin{equation}
    A=\left[\begin{matrix}s_{x} & 0 & 0,\\\beta & s_{h} & 0,\\0 & \alpha & s_{y}\end{matrix}\right],
\end{equation}
the diagonal diffusion matrix is 
\begin{equation}
    D=\left[\begin{matrix}d_h & 0 & 0,\\0 & d_{x} & 0,\\0 & 0 & d_{y}\end{matrix}\right]\in PDD_3,
\end{equation}
and the $t$-faithful covariance matrix is 
\begin{equation}
    \Sigma=\left[\begin{matrix}\sigma_{hh} & \sigma_{hx} & \sigma_{hy},\\\sigma_{hx} & \sigma_{xx} & \sigma_{xy},\\\sigma_{hy} & \sigma_{xy} & \sigma_{yy}\end{matrix}\right]\in\mathcal{M}^p_{G,\alpha}.
\end{equation}
The resulting set of equations to solve is
\begin{equation}     \begin{split}
(i)\;& - d_{h} = 2 s_{h} \sigma_{hh}  ,\\
(ii)\;& 0 = \beta \sigma_{hh} + s_{x} \sigma_{hx} + s_{h} \sigma_{hx}  ,\\
(iii)\;& - d_{x} = 2 \beta \sigma_{hx} + 2 s_{x} \sigma_{xx}  ,\\
(iv)\;& 0 = \alpha \sigma_{hx} + s_{y} \sigma_{hy} + s_{h} \sigma_{hy}  ,\\
(v)\;& 0 = \alpha \sigma_{xx} + \beta \sigma_{hy} + s_{x} \sigma_{xy} + s_{y} \sigma_{xy}  ,\\
(vi)\;& - d_{y} = 2 \alpha \sigma_{xy} + 2 s_{y} \sigma_{yy}.
\end{split} \end{equation}

Analogous to the proof in~\ref{sec:proof_no_laten_cause_n_effect}, Eq.~(iv) is satisfied if and only if $\alpha = b_1 \sigma_{hy}/\sigma_{hx}$, where $b_1 = -\big(s_y + s_h\big)$. Since $s_y,s_h<0$, we have that $b_1>0$. Therefore, $\operatorname{sign}(\alpha)=\operatorname{sign}(b_1\sigma_{hy})/\operatorname{sign}(\sigma_{hx})=\operatorname{sign}(\sigma_{hy})/\operatorname{sign}(\sigma_{hx})$. Since $\Sigma\in \mathcal{M}^p_{G,\alpha}$, we have $\sigma_{hy}\neq0$ meaning that the sign of $\alpha$ is $+$ or $-$. Therefore there exists no $\Sigma\in\mathcal{M}^0_{G,\alpha}$. According to the $\mathcal{M}^0$ criterion Theorem~\ref{thm:m0-criterion}, the sign of edge $\alpha$ in graph $G$ is identifiable.
\end{proof}

\subsubsection{Confounding}\label{sec:appendix_proof_no_latent_confoundin}
\begin{proof}
    We adopt the assumptions and conventions stated at the start of this section. Let $G=(V,E)$ be the graph of Fig.~\ref{fig:confounding}. The nodes $V=\{H,X,Y\}$ correspond to the SDE process $X=(H,X,Y)^T$, then the Hurwitz stable drift matrix $A$ respecting the causal structure of graph $G$ is
\begin{equation}
    A=\left[\begin{matrix}s_{h} & 0 & 0,\\\gamma & s_{x} & 0,\\\delta & \alpha & s_{y}\end{matrix}\right],
\end{equation}
the diagonal diffusion matrix is 
\begin{equation}
    D=\left[\begin{matrix}d_h & 0 & 0,\\0 & d_{x} & 0,\\0 & 0 & d_{y}\end{matrix}\right]\in PDD_3,
\end{equation}
and the $t$-faithful covariance matrix is 
\begin{equation}
    \Sigma=\left[\begin{matrix}\sigma_{hh} & \sigma_{hx} & \sigma_{hy},\\\sigma_{hx} & \sigma_{xx} & \sigma_{xy},\\\sigma_{hy} & \sigma_{xy} & \sigma_{yy}\end{matrix}\right]\in\mathcal{M}^p_{G,\alpha}.
\end{equation}

In the numerical Section~\ref{sec:numerical_edge} we find examples of $\Sigma,\Sigma'\in\mathcal{M}^p_{G,\alpha}$ where $\Sigma$ is identifiable, and $\Sigma'$ is non-identifiable. In other words, we show there exist covariance matrices in both $\mathcal{M}^\pm_{G,\alpha}$ and only in either $\mathcal{M}^+_{G,\alpha}$ or $\mathcal{M}^-_{G,\alpha}$. Hence, by Definition~\ref{def:edge_identifiability}, the sign of edge $\alpha$ for graph $G$ is partially identifiable.

To exclude that these examples are some degenerate cases, we show that the set of covariance matrices $\Sigma$ yielding identifiability (respectively, non-identifiability) is not a measure-zero subset in $\mathcal{M}^p_{G,\alpha}$. To that end, we first characterize $\Sigma\in\mathcal{M}^p_{G,\alpha}$. This means that we want to show that there exists a Hurwitz drift matrix $A$ with $\operatorname{supp}(A)=E$ and a diagonal diffusion matrix $D\in PDD_3$ such that the Lyapunov equation (Eq.~\eqref{eq:lyapunov}) is satisfied.  

From \cite{dettling2023identifiability} Corollary~5.4, we know that for the set
 \begin{equation}\label{eq:lem_fg_proof_mgd_set}
     \mathcal{M}_{G,D}:=\set{\Sigma \in PD_d: \exists A \text{ such that }  A\Sigma+\Sigma A^T=-D; \operatorname{supp}(A)\subseteq E},
\end{equation}
we have $\mathcal{M}_{G,D}=PD_d$ for the graph $G$ and any given $D\in PD_d$. Comparing $\mathcal{M}_{G,D}$ with the edge signature set (Definition~\ref{def:edge_signature_set}) and the resulting possibility set (Definition~\ref{def:possible_edge_set}), we observe that they are structurally similar. If the result $\mathcal{M}_{G,D}=PD_d$ could be strengthened to require $\operatorname{supp}(A)=E$, then, since we allow any $D\in PDD_d$ and $F_G\subset PD_d$, it would follow that
\begin{equation}
\mathcal{M}^p_{G,\alpha}=F_G.
\end{equation}

To show this, we start from the set of equations resulting from the Lyapunov equation.
\begin{equation}\label{eq:proof_confonounder_lypaunov_general}
\begin{split}
(i)\;& -d_h  =  2 s_h \sigma_{hh} ,\\
(ii)\;& 0=\gamma \sigma_{hh} + (s_h+s_x) \sigma_{hx}  ,\\
(iii)\;& -d_x =2 \gamma \sigma_{hx} + 2 s_x \sigma_{xx}  ,\\
(iv)\;& 0=\alpha \sigma_{hx} + \delta \sigma_{hh} + (s_h+s_y) \sigma_{hy}  ,\\
(v)\;& 0=\alpha \sigma_{xx} + \delta \sigma_{hx} + \gamma \sigma_{hy} + (s_x+s_y) \sigma_{xy}  ,\\
(vi)\;& -d_y=2 \alpha \sigma_{xy} + 2 \delta \sigma_{hy} + 2 s_y \sigma_{yy}.
\end{split}
\end{equation}

Since $d_h,\sigma_{hh}>0$ and $s_h<0$, we have that Eq.~$(i)$ is always satisfied.  Eq.~(ii) is satisfied if and only if, $\gamma = -(s_h+s_x)\sigma_{hx}/\sigma_{hh}$. Eq.~$(iii)$ is satisfied in and only if,
\begin{equation}\label{eq:proof_confouding_iii_derivation}
    \begin{split}
 2 \frac{-(s_h+s_x)}{\sigma_{hh}} \sigma_{hx}^2 + 2 s_x \sigma_{xx} &= -d_x ,\\
\iff2 \frac{-(s_h+s_x)}{\sigma_{hh}}\sigma_{hx}^2 + 2 s_x \sigma_{xx} &\overset{(a)}{<} 0,\\
\frac{-(s_h+s_x)}{\sigma_{hh}} \sigma_{hx}^2 + s_x \sigma_{xx} &<0 ,\\
\frac{-(s_h+s_x)}{\sigma_{hh}} \rho_{hx}^2 \sigma_{hh} \sigma_{xx}+s_x \sigma_{xx}&\overset{(b)}{<}0  ,\\
-(s_h+s_x)\rho_{hx}^2 \sigma_{xx} +s_x \sigma_{xx}&<0  ,\\
-(s_h+s_x)\rho_{hx}^2+s_x&<0,\\
s_x(1-\rho_{hx}^2) &< s_h\rho_{hx}^2,\\
s_x &< \frac{s_h\rho_{hx}^2}{1-\rho_{hx}^2}
= -\,\frac{s_h\rho_{hx}^2}{\rho_{hx}^2-1},
    \end{split}
\end{equation}
where $(a)$ we us that $d_x>0$ and $(b)$ we substitute $\sigma_{hx}^2=\rho_{hx}^2 \sigma_{hh} \sigma_{xx}$. 
Since $s_x,s_h<0$ and $\rho_{hx}^2<1$, we have $1-\rho_{hx}^2>0$, and hence $\rho_{hx}^2/\big(\rho_{hx}^2-1\big)>0$. Therefore, the inequality demands that
\begin{equation}
    \left|-\frac{s_h \rho_{hx}^2}{\rho_{hx}^2-1}\right| < |s_x|,
\end{equation}
and we can rewrite it as
\begin{equation}
s_x=-b_1 s_h\frac{\rho_{hx}^2}{\rho_{hx}^2-1},
\qquad b_1>1.
\end{equation}

Eq.~$(iv)$ is satisfied if and only if
\begin{equation}
    \alpha=\frac{- \delta \sigma_{hh} - (s_{h}+ s_{y}) \sigma_{hy}}{\sigma_{hx}}.
\end{equation}

Eq.~$(v)$ is satisfied if and only if 
\begin{equation}
    \begin{split}
        -\delta\sigma_{hx}&=\alpha \sigma_{xx} + \gamma \sigma_{hy} + (s_x+s_y) \sigma_{xy} ,\\
         -\delta\sigma_{hx}&\overset{(a)}{=}\frac{- \delta \sigma_{hh} - (s_{h}+ s_{y}) \sigma_{hy}}{\sigma_{hx}} \sigma_{xx} + \gamma \sigma_{hy} + (-b_1 s_h\frac{\rho_{hx}^2}{\rho_{hx}^2-1}+s_y) \sigma_{xy} ,\\
         \delta\big(\sigma_{hh}\sigma_{xx}/\sigma_{hx}-\sigma_{hx}\Big)&=\frac{- (s_{h}+ s_{y}) \sigma_{hy}}{\sigma_{hx}} \sigma_{xx} + \gamma \sigma_{hy} + (-b_1 s_h\frac{\rho_{hx}^2}{\rho_{hx}^2-1}+s_y) \sigma_{xy} ,\\
         \delta&=\big(\frac{- (s_{h}+ s_{y}) \sigma_{hy}}{\sigma_{hx}} \sigma_{xx} + \gamma \sigma_{hy} + (-b_1 s_h\frac{\rho_{hx}^2}{\rho_{hx}^2-1}+s_y) \sigma_{xy}\big)/\big(\sigma_{hh}\sigma_{xx}/\sigma_{hx}-\sigma_{hx}\Big) ,\\
         \delta&=\big(- (s_{h}+ s_{y}) \sigma_{hy} \sigma_{xx} + \gamma \sigma_{hy}\sigma_{hx} + (-b_1 s_h\frac{\rho_{hx}^2}{\rho_{hx}^2-1}+s_y) \sigma_{xy}\sigma_{hx}\big)/\big(\sigma_{hh}\sigma_{xx}-\sigma_{hx}^2\Big) ,\\
    \end{split}
\end{equation}
where $(a)$ we substituted the values for $\alpha$ and $s_x$.

Since $d_y>0$, Eq.~$(vi)$ is satisfied if and only if 
\begin{equation}
    \begin{split}
         0&>2 \alpha \sigma_{xy} + 2 \delta \sigma_{hy} + 2 s_y\sigma_{yy},\\
         0&> \alpha \sigma_{xy} + \delta \sigma_{hy} +  s_y\sigma_{yy},\\
         0&\overset{(a)}{>} \frac{- \delta \sigma_{hh} - (s_{h}+ s_{y}) \sigma_{hy}}{\sigma_{hx}}\sigma_{xy} + \delta \sigma_{hy} +  s_y\sigma_{yy},\\
          0&>  \delta \big(\sigma_{hy} -\frac{\sigma_{hh}\sigma_{xy} }{\sigma_{hx}}\sigma_{xy}\big)
          -(s_{h}+ s_{y})\frac{\sigma_{xy} \sigma_{hy}}{\sigma_{hx}}+  s_y\sigma_{yy},\\
          0&\overset{(b)}{>} \big(- (s_{h}+ s_{y}) \sigma_{hy} \sigma_{xx} + \gamma \sigma_{hy}\sigma_{hx} + (-b_1 s_h\frac{\rho_{hx}^2}{\rho_{hx}^2-1}+s_y) \sigma_{xy}\sigma_{hx}\big)\big( \sigma_{hy}-\frac{\sigma_{hh}\sigma_{xy} }{\sigma_{hx}}\big)/\big(\sigma_{hh}\sigma_{xx}-\sigma_{hx}^2\Big) \\
          &- (s_{h}+ s_{y}) \sigma_{hy} +  s_y\sigma_{yy},\\
          0&\overset{(c)}{>}f(\Sigma,b_1,s_h,s_y).          
    \end{split}
\end{equation}
where $(a)$ we substituted the values for $\alpha$, $(b)$ we substituted the value of $\delta$ and $(c)$ we introduced the function $f$ to denote the lengthy the right hand side.

In addition, we can write the expressions for $\alpha,\gamma$ and $\delta$ in a different way when starting from Eq.~$(ii,iv)$ and $(v)$.
We begin by substituting $s_x$ in to $\gamma$, then
\begin{equation}
   \begin{split}
        \gamma&=-(s_h+-b_1 s_h\frac{\rho_{hx}^2}{\rho_{hx}^2-1})\frac{\sigma_{hx}}{\sigma_{hh}} ,\\
        &=-(1+-b_1\frac{\rho_{hx}^2}{\rho_{hx}^2-1}) s_h\frac{\sigma_{hx}}{\sigma_{hh}}, \\
        &\overset{(a)}{=}(1-t\frac{\rho_{hx}^2}{\rho_{hx}^2-1}) t\frac{\sigma_{hx}}{\sigma_{hh}},
   \end{split}
\end{equation}
 where $(a)$ we pick $b_1=t$ and $s_h=-t$, with $t>1$ satisfying both $s_h<0$ and $b_1>1$, then $\gamma$ only depends on $t$ and we can write $\gamma(t)$. Furthermore, since $t,\sigma_{hh}>0$ and $\sigma_{hx}\neq0$, $\gamma=0$ if and only if $t=\big(\rho_{hx}^2-1\big)/\rho_{hx}^2$. With this choice of $b_1=t$ and $s_h=-t$, $s_x$ also only depends on $t$, i.e., 
 \begin{equation}
     s_x=t^2\frac{\rho_{hx}^2}{\rho_{hx}^2-1}.
 \end{equation}
We then pick $s_y=-(t+\varepsilon)$, where $\varepsilon>0$.
 
Starting from Eq.~\eqref{eq:proof_confonounder_lypaunov_general}, we can write Eq.~$(iv)$ and $(v)$ as a system of linear equations
\begin{equation}\label{eq:confounding_2x2_system}
\begin{pmatrix}
\sigma_{hh} & \sigma_{hx}\\
\sigma_{hx} & \sigma_{xx}
\end{pmatrix}
\begin{pmatrix}
\delta\\
\alpha
\end{pmatrix}
=
\begin{pmatrix}
-(s_h+s_y)\sigma_{hy}\\
-\,\gamma\sigma_{hy}-(s_x+s_y)\sigma_{xy}
\end{pmatrix}.
\end{equation}

Moreover, Eq.~in~\eqref{eq:confounding_2x2_system} the coefficient matrix
\begin{equation}
M:=\begin{pmatrix}
\sigma_{hh} & \sigma_{hx}\\
\sigma_{hx} & \sigma_{xx}
\end{pmatrix}
\end{equation}
does not depend on $\varepsilon$, whereas the right-hand side does. 
Using our choices of $s_h,\,s_x$ and $s_y$ we have
\begin{equation}
-(s_h+s_y)\sigma_{hy}=(2t+\varepsilon)\sigma_{hy},
\qquad
-(s_x+s_y)\sigma_{xy}=(-t^2\frac{\rho_{hx}^2}{\rho_{hx}^2-1}+t+\varepsilon)\sigma_{xy}.
\end{equation}
Moreover, for $\gamma(t)$ we obtained
\begin{equation}
\gamma=(1-t\frac{\rho_{hx}^2}{\rho_{hx}^2-1}) t\frac{\sigma_{hx}}{\sigma_{hh}}=\big(t-t^2\frac{\rho_{hx}^2}{\rho_{hx}^2-1}\big) \frac{\sigma_{hx}}{\sigma_{hh}},
\end{equation}
which is independent of $\varepsilon$. Hence the right-hand side of Eq.~\eqref{eq:confounding_2x2_system} equals
\begin{align}
r(\varepsilon)
&=\nonumber
\begin{pmatrix}
-(s_h+s_y)\sigma_{hy}\\
-\,\gamma\sigma_{hy}-(s_x+s_y)\sigma_{xy}
\end{pmatrix}\\
&=\nonumber
\begin{pmatrix}
(2t+\varepsilon)\sigma_{hy}\\[2pt]
-\;\big(t-t^2\frac{\rho_{hx}^2}{\rho_{hx}^2-1}\big) \frac{\sigma_{hx}}{\sigma_{hh}}\sigma_{hy}+\big(t-t^2\frac{\rho_{hx}^2}{\rho_{hx}^2-1}+\varepsilon\big)\sigma_{xy}
\end{pmatrix}\\
&=
\underbrace{\begin{pmatrix}
2t\,\sigma_{hy}\\[2pt]
\big(t-t^2\frac{\rho_{hx}^2}{\rho_{hx}^2-1}\big)\big(\sigma_{xy}-\frac{\sigma_{hx}\sigma_{hy}}{\sigma_{hh}}\big)
\end{pmatrix}}_{=:~r_0}
\;+\;
\varepsilon\underbrace{\begin{pmatrix}
\sigma_{hy}\\[2pt]
\sigma_{xy}
\end{pmatrix}}_{=:~r_1}.
\end{align}

Therefore,
\begin{equation}
r(\varepsilon)=r_0+\varepsilon r_1,
\end{equation}
for fixed vectors $r_0,r_1\in\mathbb{R}^2$ (depending on $t$ and $\Sigma$ but not on $\varepsilon$).
Since $\Sigma\succ 0$, $M$ is positive definite.
Therefore,
\begin{equation}
\begin{pmatrix} \delta(\varepsilon)\\ \alpha(\varepsilon)\end{pmatrix}
=M^{-1}r(\varepsilon)
=M^{-1}r_0+\varepsilon\,M^{-1}r_1,
\end{equation}
so both $\delta(\varepsilon)$ and $\alpha(\varepsilon)$ are affine functions of $\varepsilon$. 

Now suppose $\delta(\varepsilon)$ were identically zero for all $\varepsilon$. Then both its constant and linear coefficients would be zero, i.e.,
$e_1^\top M^{-1}r_0=e_1^\top M^{-1}r_1=0$, which would force $(M^{-1}r_1)_1=0$.
But $r_1=(\sigma_{hy},\sigma_{xy})^\top$, and since $M$ is invertible this would imply $(\sigma_{hy},\sigma_{xy})^\top=0$, contradicting $\sigma_{hy}\neq 0$ and $\sigma_{xy}\neq 0$ (which hold for $\Sigma\in F_G$ in the confounding graph). Hence $\delta(\varepsilon)$ is a non-constant affine function and can vanish for at most one value of $\varepsilon$. The same argument applies to $\alpha(\varepsilon)$.

Consequently, there are at most two values of $\varepsilon$ for which either $\delta(\varepsilon)=0$ or $\alpha(\varepsilon)=0$. Choosing any $\varepsilon>0$ different from these values yields
\begin{equation}
\delta(\varepsilon)\neq 0
\qquad\text{and}\qquad
\alpha(\varepsilon)\neq 0.
\end{equation}

Therefore, to summarise, if we choose 
\begin{equation}
t \neq \big(\rho_{hx}^2-1\big)/\rho_{hx}^2
\end{equation}
and $\varepsilon$ different from the two corresponding conflicting values, then $\operatorname{supp}(A)=E$.

Since $f(\Sigma,-t,t,t+\varepsilon)$ is a strict inequality and a continuous function (in fact, a multivariate polynomial) in $(b_1,s_h,s_y)$, we may choose $t$ and $\varepsilon$ at positive distance $\varepsilon'>0$ from their conflicting values, while still satisfying $f(\Sigma,-t,t,t+\varepsilon)<0$. Hence, whenever $f(\Sigma,-t,t,t+\varepsilon)$ holds with $\operatorname{supp}(A)\subseteq E$, we can perturb $t$ and $\varepsilon$ slightly so that the inequality remains satisfied while ensuring $\operatorname{supp}(A)=E$.

Using the aforementioned result that $\mathcal{M}_{G,D}=PD_d$, it follows that for every $\Sigma\in PD_3$ there exists a feasible choice of parameters satisfying
\begin{equation}
    f(\Sigma,-t,t,t+\varepsilon)<0,
\end{equation}
with $\operatorname{supp}(A)\subseteq E$.

If the corresponding choice of $t$ and $\varepsilon$ yields $\operatorname{supp}(A)\subset E$, we can perturb them to $\bar t,\bar\varepsilon$ such that $\operatorname{supp}(A)=E$ while preserving the strict inequality, by continuity of $f$.
Consequently, we obtain
\begin{equation}
    \mathcal{M}^p_{G,\alpha}=F_G.
\end{equation}

Next we characterize models in $\mathcal{M}^0_{G,\alpha}$. Note that now, since $\alpha=0$, we only require $\operatorname{supp}(A)\subset E$. We set $\alpha=0$, which simplifies Eq. $(iv),(v)$ and $(vi)$ to
\begin{equation}
\begin{split}
(iv)\;& \delta \sigma_{hh} + (s_h+s_y) \sigma_{hy} = 0
,\\
(v)\;& \delta \sigma_{xx} + \gamma \sigma_{hy} + (s_x+s_y) \sigma_{xy} = 0 ,\\
(vi)\;& 2 \delta \sigma_{hy} + 2 s_y \sigma_{yy} = -d_y,
\end{split}
\end{equation}
The other Eq.~$(i-iii)$ remain the same as before, meaning that Eq.~$(i)$ is always satisfied, Eq.~$(ii)$ is satisfied if and only if $\gamma = -(s_h+s_x)/\sigma_{hh} \sigma_{hx}$ and  Eq.~
$(iii)$ is satisfied if and only if $s_x=-b_1s_h\rho_{hx}^2/\big(\rho_{hx}^2-1\big)$ with $b_1>1$.

From the new set of equations with $\alpha=0$, we see that Eq.~(iv) is satisfied if and only if $\delta = b_2 \sigma_{hy}$ where $b_2 = -(s_h+s_y)/\sigma_{hh}>0$.

Therefore, it remains to satisfy the two Eq. $(v)$ and $(vi)$. In these equations, we substitute $\delta=b_2 \sigma_{hy},\ \gamma=b_1 \sigma_{hx}$.

Eq. $(v)$ becomes
\begin{equation}
    \begin{split}
 b_2 \sigma_{hy}\sigma_{hx} + b_1 \sigma_{hx}\sigma_{hy} + (s_x+s_y)\sigma_{xy} &= 0  ,\\
-(s_h+s_y)\rho_{hx}\rho_{hy}\sqrt{\sigma_{xx}\sigma_{yy}}
-(s_h+s_x)\rho_{hx}\rho_{hy}\sqrt{\sigma_{xx}\sigma_{yy}}
+(s_x+s_y)\rho_{xy}\sqrt{\sigma_{xx}\sigma_{yy}}&\overset{(a)}{=}0 ,\\
\Big[-(2s_h+s_x+s_y)\rho_{hx}\rho_{hy}
+(s_x+s_y)\rho_{xy}\Big]\sqrt{\sigma_{xx}\sigma_{yy}}&=0 ,\\
-(2s_h+s_x+s_y)\rho_{hx}\rho_{hy}+(s_x+s_y)\rho_{xy}&=0,
    \end{split}
\end{equation}
where $(a)$ $b_2 = -\big(s_h+s_y\big)/\sigma_{hh},
\, b_1=-(s_h+s_x)/\sigma_{hh}$ and $\sigma_{ij}=\rho_{ij}\sqrt{\sigma_{ii}\sigma_{jj}}$.
Note that in the final line, due to $-(2s_h+s_x+s_y)>0$ and $(s_x+s_y)<0$, $\,\operatorname{sign}(\rho_{hx}\rho_{hy})\ne\operatorname{sign}(\rho_{xy})$ leads to a contradiction. Therefore, to be in $\mathcal{M}^0_{G,\alpha}$, we should have $\operatorname{sign}(\sigma_{hx}\sigma_{hy})=\operatorname{sign}(\sigma_{xy})$.

Eq. $(vi)$ becomes
\begin{equation}
    \begin{split}
2 b_2 \sigma_{hy}^2 + 2 s_y \sigma_{yy} &= -d_y ,\\
-(s_h+s_y)\rho_{hy}^2+s_y&\overset{(a)}{<}0
,
    \end{split}
\end{equation}
where $(a)$ uses the same steps as in Eq.~\eqref{eq:proof_confouding_iii_derivation}, $\,b_2 = -\big(s_h+s_y\big)/\sigma_{hh}$ and $\sigma_{ij}=\rho_{ij}\sqrt{\sigma_{ii}\sigma_{jj}}$. Which analogously yields
\begin{equation}
s_y=-b_3 s_h\frac{\rho_{hy}^2}{\rho_{hy}^2-1},
\end{equation}
with $b_3>1$.

Summarising, $s_y$ and $s_x$ can be expressed as
\begin{equation}
    s_y=-b_3 s_h\frac{\rho_{hy}^2}{\rho_{hy}^2-1},
\qquad\text{and}\qquad
s_x=-b_1 s_h\frac{\rho_{hx}^2}{\rho_{hx}^2-1},
\end{equation}
with $b_1,b_3>1$

Substituting this into Eq. $(v)$ yields
\begin{equation}\label{eq:proof_confounding_negative_requirement}
    \begin{split}
    0&=-\Big(2s_h-b_3 s_h\frac{\rho_{hy}^2}{\rho_{hy}^2-1}
-b_1 s_h\frac{\rho_{hx}^2}{\rho_{hx}^2-1}\Big)\rho_{hx}\rho_{hy}
+\Big(-b_3 s_h\frac{\rho_{hy}^2}{\rho_{hy}^2-1}
-b_1 s_h\frac{\rho_{hx}^2}{\rho_{hx}^2-1}\Big)\rho_{xy} ,\\
0&=-2\rho_{hx}\rho_{hy} + \Big(b_3\frac{\rho_{hy}^2}{\rho_{hy}^2-1}
+b_1\frac{\rho_{hx}^2}{\rho_{hx}^2-1}\Big)(\rho_{hx}\rho_{hy}-\rho_{xy}) ,\\
\frac{2\rho_{hx}\rho_{hy}}{\rho_{hx}\rho_{hy}-\rho_{xy}}&=b_3\frac{\rho_{hy}^2}{\rho_{hy}^2-1}
+b_1\frac{\rho_{hx}^2}{\rho_{hx}^2-1}.
    \end{split}
\end{equation}
Since the right hand side is negative ($\rho_{ij}^2-1<0$), the left hand side must also be negative.

We can solve this equation exactly. For notational simplicity let $a:=\rho_{hy}^2/\big(\rho_{hy}^2-1\big)<0,\, b:=\rho_{hx}^2/\big(\rho_{hx}^2-1\big)<0$ and $c:=2\rho_{hx}\rho_{hy}/\big(\rho_{hx}\rho_{hy}-\rho_{xy}\big)<0$. Then from the above we have that 
\begin{equation}
    \begin{split}
     c&= b_3 a + b_1 b,\\
    \frac{-b_1 b + c}{a}&=  b_3   ,\\
    \frac{-b_1 b + c}{a} &\overset{(a)}{>} 1,\\
    -b_1 b + c &\overset{(b)}{<} a,\\
b_1 &< \frac{-a + c}{b} ,\\
 1\overset{(c)}{<}& \frac{-a + c}{b},
    \end{split}
\end{equation}
where $(a)$ enforces $b_3>1$, i.e., satisfying Eq.~$(vi)$, $(b)$ since $a<0$ the inequality $>$ flips to $<$ and $(c)$ is the tightest way to allow a choice $b_1>1$, i.e., satisfying Eq.~$(iii)$.
Substituting the original definitions for $a,b$ and $c$ back the inequality gives
\begin{equation}\label{eq:confounding_step1}
    \begin{split}
    1&<\frac{-\frac{\rho_{hy}^2}{\rho_{hy}^2-1}+\frac{2\rho_{hy}\rho_{hx}}{\rho_{hy}\rho_{hx}-\rho_{xy}}}{\frac{\rho_{hx}^2}{\rho_{hx}^2-1}},\\
    \frac{\rho_{hx}^2}{\rho_{hx}^2-1}&\overset{(a)}{>}-\frac{\rho_{hy}^2}{\rho_{hy}^2-1}+\frac{2\rho_{hy}\rho_{hx}}{\rho_{hy}\rho_{hx}-\rho_{xy}},\\
    \frac{\rho_{hy}^2}{\rho_{hy}^2-1}+\frac{\rho_{hx}^2}{\rho_{hx}^2-1}&>\frac{2\rho_{hy}\rho_{hx}}{\rho_{hy}\rho_{hx}-\rho_{xy}},\\
    \frac{\rho_{hy}^2}{\rho_{hy}^2-1}+\frac{\rho_{hx}^2}{\rho_{hx}^2-1}&\overset{(b)}{>}\frac{2d\rho_{xy}}{d\rho_{xy}-\rho_{xy}},\\
    \frac{\rho_{hy}^2}{\rho_{hy}^2-1}+\frac{\rho_{hx}^2}{\rho_{hx}^2-1}&>\frac{2d}{d-1},
    \end{split}
\end{equation}
where (a) we use that since $\rho_{hx}^2<1$, we have that the inequality flips, and (b) we use the notation $\rho_{hy}\rho_{hx}=d\rho_{xy}$. Since $\rho_{hy}^2, \rho_{hx}^2 < 1$, the left-hand side of Eq.~\eqref{eq:confounding_step1} is less than zero. Therefore, $2d/(d-1)<0$ meaning that $0<d<1$, i.e., if
\begin{equation}
    \frac{\rho_{hy}\rho_{hx}}{\rho_{xy}}\geq1 \quad \text{or} \quad \frac{\rho_{hy}\rho_{hx}}{\rho_{xy}}\leq0,
\end{equation}
then there is a contradiction. Note that $\rho_{xy}/(\rho_{hy}\rho_{hx})<0$ was already excluded by the requirement that $\operatorname{sign}(\rho_{hy}\rho_{hx})=\operatorname{sign}(\rho_{xy})$. Continuing with the derivation of Eq.~\eqref{eq:confounding_step1},

\begin{equation}\label{eq:confounding_step2}
    \begin{split}
    \frac{\rho_{hy}^2}{\rho_{hy}^2-1}+\frac{\rho_{hx}^2}{\rho_{hx}^2-1}&>\frac{2d}{d-1},\\
    \frac{\rho_{hy}^2(\rho_{hx}^2-1)+\rho_{hx}^2(\rho_{hy}^2-1)}{(\rho_{hy}^2-1)(\rho_{hx}^2-1)}&>\frac{2d}{d-1},\\
    \frac{2\rho_{hy}^2\rho_{hx}^2-\rho_{hx}^2-\rho_{hy}^2}{(\rho_{hy}^2-1)(\rho_{hx}^2-1)}&>\frac{2d}{d-1},\\
    \frac{(2\rho_{hy}^2\rho_{hx}^2-\rho_{hx}^2-\rho_{hy}^2)(d-1)}{2d(\rho_{hy}^2-1)(\rho_{hx}^2-1)}&\overset{(a)}{<}1,\\
    \frac{(2\rho_{hy}^2\rho_{hx}^2-\rho_{hx}^2-\rho_{hy}^2)(\rho_{hy}\rho_{hx}-\rho_{xy})}{2\rho_{hy}\rho_{hx}(\rho_{hy}^2-1)(\rho_{hx}^2-1)}&\overset{(b)}{<}1,
    \end{split}
\end{equation}
where in (a) we use that since $2d/(d-1)<0$, the inequality flips, and in (b) we take the reverse of the final two steps of Eq.~\eqref{eq:confounding_step1}.

Eq. $(v)$ has become a single inequality that enforces Eq. $(iii)$ and $(vi)$ as well. Therefore, satisfying this final inequality is a necessary condition for membership in $\mathcal{M}^0_{G,\alpha}$.

To collect the results, there are three conditions on $\Sigma$ for membership in $\mathcal{M}^0_{G,\alpha}$:
\begin{equation}\label{eq:proof_confounding_conditions}
    \begin{split}
    (c.1)&\qquad \frac{(2\rho_{hy}^2\rho_{hx}^2-\rho_{hy}^2-\rho_{hx}^2)(\rho_{hx}\rho_{hy}-\rho_{xy})}{2\rho_{hx}\rho_{hy}(\rho_{hx}^2-1)(\rho_{hy}^2-1)}<1, \\
    (c.2)&\quad \operatorname{sign}(\sigma_{hx}\sigma_{hy})=\operatorname{sign}(\sigma_{xy}), \\
    (c.3)&\quad \frac{\rho_{hy}\rho_{hx}}{\rho_{xy}}<1.
    \end{split}
\end{equation}

Hence, a covariance matrix $\Sigma$ belongs to $\mathcal{M}^0_{G,\alpha}$ if and only if it satisfies all the conditions in Eq.~\eqref{eq:proof_confounding_conditions}. By the $\mathcal{M}^0$-criterion Theorem~\ref{thm:m0-criterion}, we have that covariance matrices $\Sigma\in \mathcal{M}^p_{G,\alpha}$ violating one of the conditions in Eq.~\eqref{eq:proof_confounding_conditions} are identifiable. 

To summarise what we have shown up to this point, $\mathcal{M}^p_{G,\alpha}=F_G$, and verifying if $\Sigma\in\mathcal{M}^0_{G,\alpha}$ reduces to three conditions. Hence, any $\Sigma\in F_G$ that satisfies the three conditions is in $\mathcal{M}^0_{G,\alpha}$ and is thus non-identifiable. Conversely, any $\Sigma\in F_G$ that violates one of the three conditions is identifiable. We will use this to show that the sets of non-identifiable and identifiable covariance matrices both have non-zero measure. Throughout the rest of the proof, we will use the notation that for a $\rho_{ij}\in(\underline{\rho}_{ij},\overline{\rho}_{ij})$ we have that $\underline{\rho}_{ij}$ is the infimum and $\overline{\rho}_{ij}$ is the supremum.

To begin, let
\begin{equation}
    \Sigma_{id}:=\{\Sigma\in F_G:\,\Sigma \text{ does not satisfy one of the conditions in Eq.}\eqref{eq:proof_confounding_conditions}\}.
\end{equation}
Let $\Sigma_{\mathrm{set}}$ be a set of covariance matrices for which the correlation coefficients are
\begin{equation}
\rho_{hx}\in(0.001,0.002),\quad \rho_{hy}\in(-0.001,0),\quad \rho_{xy}\in(0.001,0.002).
\end{equation}
Then
\begin{equation}
\begin{split}
    1 + 2\rho_{hx}\rho_{hy}\rho_{xy}
    -(\rho_{hx}^2+\rho_{hy}^2+\rho_{xy}^2)
    &\ge 1 + 2\overline{\rho}_{hx}\underline{\rho}_{hy}\overline{\rho}_{xy}
    -(\overline{\rho}_{hx}^2+\underline{\rho}_{hy}^2+\overline{\rho}_{xy}^2),\\
    &= 1 + 2\cdot0.002\cdot (-0.001)\cdot 0.002
    - (0.002^2 + (-0.001)^2 + 0.002^2) ,\\
    &\approx 1.00>0.
\end{split}
\end{equation}
Hence, Sylvester's criterion Eq.~\eqref{eq:sylvester_criterion_3__withtout_latent} holds. Therefore, since $|\rho_{ij}|<1$, we have that $\Sigma\in PD_3$. Moreover, the non-zero correlations $\rho_{ij}$ respect the marginal independences of $G$, hence $\Sigma_{\mathrm{set}}\in F_G$.

Furthermore, since $\operatorname{sign}(\sigma_{ij})=\operatorname{sign}(\rho_{ij})$, we have
\begin{equation}
    \operatorname{sign}(\sigma_{hx}\sigma_{hy})=-\neq+=\operatorname{sign}(\sigma_{xy}),
\end{equation}
so $\Sigma_{\mathrm{set}}$ violates Condition~$(c.2)$ from Eq.~\eqref{eq:proof_confounding_conditions}. Hence, $\Sigma_{\mathrm{set}}\subseteq \Sigma_{id}$.

Finally, the Lebesgue measure satisfies
\begin{equation}
    m(\Sigma_{\mathrm{set}}) = 0.001^3> 0.
\end{equation}
By monotonicity, we conclude
\begin{equation}
    m(\Sigma_{id}) \ge m(\Sigma_{\mathrm{set}}) > 0,
\end{equation}
and hence $\Sigma_{id}$ is not a measure-zero set.

Now let
\begin{equation}
    \Sigma_{non}:=\{\Sigma\in F_G:\,\Sigma \text{ satisfies the conditions in Eq.}\eqref{eq:proof_confounding_conditions}\},
\end{equation}
and let $\Sigma_{\mathrm{set}}$ be a set of covariance matrices with the correlation coefficients
\begin{equation}
\rho_{hx}\in(0.326,0.338),\quad \rho_{hy}\in(0.29,0.305),\quad \rho_{xy}\in(0.734,0.842).
\end{equation}
Then
\begin{equation}
\begin{split}
    1 + 2\rho_{hx}\rho_{hy}\rho_{xy}
    -(\rho_{hx}^2+\rho_{hy}^2+\rho_{xy}^2)
    &\ge 1 + 2\underline{\rho}_{hx}\underline{\rho}_{hy}\underline{\rho}_{xy}
    -(\overline{\rho}_{hx}^2+\overline{\rho}_{hy}^2+\overline{\rho}_{xy}^2),\approx 0.22>0.
\end{split}
\end{equation}
Hence,  Sylvester's criterion Eq.~\eqref{eq:sylvester_criterion_3__withtout_latent} holds. Therefore, since $|\rho_{ij}|<1$, we have that $\Sigma\in PD_3$.  Moreover, the non-zero correlations $\rho_{ij}$ respect the marginal independences of $G$, hence $\Sigma_{\mathrm{set}}\in F_G$.

To satisfy the conditions in Eq.~\eqref{eq:proof_confounding_conditions}, we
require that in  Condition $(c.1)$ the left-hand side is smaller than one on the whole domain, i.e.,
\begin{equation}\label{eq:proof_confounding_reference_latent}
   \begin{split}
        \operatorname{max}\big(\frac{(2\rho_{hy}^2\rho_{hx}^2-\rho_{hy}^2-\rho_{hx}^2)(\rho_{hx}\rho_{hy}-\rho_{xy})}{2\rho_{hx}\rho_{hy}(\rho_{hx}^2-1)(\rho_{hy}^2-1)}\big)&<1,\\
        \operatorname{max}\big(\frac{(2\rho_{hy}^2\rho_{hx}^2-\rho_{hy}^2-\rho_{hx}^2)(d-1)}{2d(\rho_{hx}^2-1)(\rho_{hy}^2-1)}\big)&<1,
   \end{split}
\end{equation}
where we use the same step as in Eq.~\eqref{eq:confounding_step1} with the notation that $\rho_{hy}\rho_{hx}=d\rho_{xy}$ and $0<d<1$. We will show that the numerator and denominator are always positive. Therefore, to find the maximum, we need to get the numerator as large as possible and the denominator as small as possible. 

To begin with the first term in the numerator. Since $ \rho_{hy}^2,\rho_{hx}^2\in(0,1)$, we have that $\rho_{hy}^2-1<0,$\, and \, $\rho_{hx}^2-1<0$. Hence,
\begin{equation}
    2\rho_{hy}^2\rho_{hx}^2-\rho_{hy}^2-\rho_{hx}^2=\rho_{hy}^2\big(\rho_{hx}^2-1\big)+\rho_{hx}^2\big(\rho_{hy}^2-1\big)<0.
\end{equation}
Therefore,
\begin{equation}
    |2\rho_{hy}^2\rho_{hx}^2-\rho_{hy}^2-\rho_{hx}^2|<|2\underline{\rho}_{hy}^2\underline{\rho}_{hx}^2-\overline{\rho}_{hy}^2-\overline{\rho}_{hx}^2|.
\end{equation}

For the second term in the numerator, since $0<d<1$, we have that
\begin{equation}
    d-1<1.
\end{equation}
Therefore, both terms in the numerator are negative, making the whole numerator positive.
In addition, since $\rho_{hy}\rho_{hx}=d\rho_{xy}$, we have that 
\begin{equation}
    |d-1|\leq |\operatorname{min}(d)-1|=|\frac{\underline{\rho}_{hx}\underline{\rho}_{hy}}{\overline{\rho}_{xy}}-1|.
\end{equation}

To continue with the first term in the denominator. We have that
\begin{equation}
    d\geq \operatorname{min}(d)=\frac{\underline{\rho}_{hx}\underline{\rho}_{hy}}{\overline{\rho}_{xy}}\approx0.11>0.
\end{equation}

For the second and third term, since $|\rho_{ij}|<1$, we have that
\begin{equation}
    \rho_{ij}^2-1<0.
\end{equation}
Therefore,
\begin{equation}
    (\rho_{hx}^2-1)(\rho_{hy}^2-1)>0.
\end{equation}
Hence, the denominator is positive. Furthermore, since 
\begin{equation}
    |\rho_{ij}^2-1|\geq |\overline{\rho}_{ij}^2-1|,
\end{equation}
we have that
\begin{equation}
    |(\rho_{hx}^2-1)(\rho_{hy}^2-1)|\geq |(\overline{\rho}_{hx}^2-1)(\overline{\rho}_{hy}^2-1)|.
\end{equation}

Combining it all together, we have
\begin{equation}
   \begin{split}
        \operatorname{max}\big(\frac{(2\rho_{hy}^2\rho_{hx}^2-\rho_{hy}^2-\rho_{hx}^2)(d-1)}{2d(\rho_{hx}^2-1)(\rho_{hy}^2-1)}\big)&<\big(\frac{(2\underline{\rho}_{hy}^2\underline{\rho}_{hx}^2-\overline{\rho}_{hy}^2-\overline{\rho}_{hx}^2)(\operatorname{min}(d)-1)}{2\operatorname{min}(d)(\overline{\rho}_{hx}^2-1)(\overline{\rho}_{hy}^2-1)}\approx0.93<1,
   \end{split}
\end{equation}
hence, satisfying Condition~$(c.1)$.

Finally, we verify the other two conditions in Eq.~\eqref{eq:proof_confounding_conditions}. For Condition~$(c.2)$, we note that all the correlation coefficients of the covariance matrices in $\Sigma_{\mathrm{set}}$ are positive, and hence Condition~$(c.2)$ is satisfied. For Condition~$(c.3)$ we note that,
\begin{equation}
    \frac{\rho_{hy}\rho_{hx}}{\rho_{xy}}\leq\frac{\overline{\rho}_{hy}\overline{\rho}_{hx}}{\underline{\rho}_{xy}}\approx0.14<1.
\end{equation}
Meaning that Condition~$(c.3)$ is also satisfied.

Thus $\Sigma_{\mathrm{set}}$ satisfies the conditions in Eq.~\eqref{eq:proof_confounding_conditions}, and therefore $\Sigma_{\mathrm{set}}\subseteq \Sigma_{non}$.

Finally, the Lebesgue measure satisfies
\begin{equation}
    m(\Sigma_{\mathrm{set}}) \approx 0.01\cdot 0.02\cdot 0.11 > 0.
\end{equation}
By monotonicity, we conclude
\begin{equation}
    m(\Sigma_{non}) \ge m(\Sigma_{\mathrm{set}}) > 0,
\end{equation}
and hence $\Sigma_{non}$ is not a measure-zero set.

Therefore, both the identifiable and non-identifiable covariance matrices form subsets of $\mathcal{M}^p_{G,\alpha}$ with positive Lebesgue measure. Hence, the edge $\alpha$ in graph $G$ is partially identifiable with positive measure.

\end{proof}

\subsubsection{Cycle of Length 3}\label{sec:no_laten}
\begin{proof}
     We adopt the assumptions and conventions stated at the start of this section. Let $G=(V,E)$ be the graph of Fig.~\ref{fig:three_loop}. The nodes $V=\{H,X,Y\}$ correspond to the SDE process $X=(H,X,Y)^T$, then the Hurwitz stable drift matrix $A$ respecting the causal structure of graph $G$ is
\begin{equation}
    A=\left[\begin{matrix}s_{x} & 0 & \gamma,\\\beta & s_{h} & 0,\\0 & \alpha & s_{y}\end{matrix}\right],
\end{equation}
the diagonal diffusion matrix is 
\begin{equation}
    D=\left[\begin{matrix}d_h & 0 & 0,\\0 & d_{x} & 0,\\0 & 0 & d_{y}\end{matrix}\right]\in PDD_3,
\end{equation}
and the $t$-faithful covariance matrix is 
\begin{equation}
    \Sigma=\left[\begin{matrix}\sigma_{hh} & \sigma_{hx} & \sigma_{hy},\\\sigma_{hx} & \sigma_{xx} & \sigma_{xy},\\\sigma_{hy} & \sigma_{xy} & \sigma_{yy}\end{matrix}\right]\in\mathcal{M}^p_{G,\alpha}.
\end{equation}

In the numerical Section~\ref{sec:numerical_edge} we find examples of $\Sigma,\Sigma'\in\mathcal{M}^p_{G,\alpha}$ where $\Sigma$ is identifiable, and $\Sigma'$ is non-identifiable. Therefore we show there exist covariance matrices in both $\mathcal{M}^\pm_{G,\alpha}$ and only in either $\mathcal{M^+_{G,\alpha}}$ or $\mathcal{M^+_{G,\alpha}}$ such that the edge $\alpha$ for graph $G$ is partially-identifiable.

For completeness, we want to show when $\Sigma\in\mathcal{M}^p_{G,\alpha}$ is (non-)identifiable.
The resulting set of equations to solve is
\begin{equation}     \begin{split}
(i)\;& - d_{x} = 2 \gamma \sigma_{xy} + 2 s_{x} \sigma_{xx}  ,\\
(ii)\;& 0 = \beta \sigma_{xx} + \gamma \sigma_{hy} + s_{h} \sigma_{xh} + s_{x} \sigma_{xh}  ,\\
(iii)\;& - d_{h} = 2 \beta \sigma_{xh} + 2 s_{h} \sigma_{hh} ,\\
(iv)\;& 0 = \alpha \sigma_{xh} + \gamma \sigma_{yy} + s_{y} \sigma_{xy} + s_{x} \sigma_{xy}  ,\\
(v)\;& 0 = \alpha \sigma_{hh} + \beta \sigma_{xy} + s_{h} \sigma_{hy} + s_{y} \sigma_{hy}  ,\\
(vi)\;& - d_{y} = 2 \alpha \sigma_{hy} + 2 s_{y} \sigma_{yy}.
\end{split} \end{equation}

We note that since the drift matrix $A$ is not triangular, the self loops $s_x,s_y$ and $s_h$ are unconstrained. 

In order to characterize models in $\mathcal{M}^0_{G,\alpha}$, we set $\alpha=0$, which simplifies Eq.~$(iv),(v)$ and $(vi)$ to
\begin{equation}     \begin{split}
(iv)\;& 0 = \gamma \sigma_{yy} + s_{y} \sigma_{xy} + s_{x} \sigma_{xy}  ,\\
(v)\;& 0 = \beta \sigma_{xy} + s_{h} \sigma_{hy} + s_{y} \sigma_{hy}  ,\\
(vi)\;& - d_{y} = 2 s_{y} \sigma_{yy}.
\end{split} \end{equation}

Due to $d_y,\sigma_{yy}>0$, Eq.~$(vi)$ is satisfied if and only if $s_y<0$. Moreover, Eq.~$(v)$ is satisfied if and only if 
\begin{equation}     \begin{split}
    \beta& = -\frac{(s_h+s_y)\sigma_{hy}}{\sigma_{xy}},\\ &\overset{(a)}{=} -(s_h+s_y)\frac{\rho_{hy}}{\rho_{xy}}\sqrt{\frac{\sigma_{hh}\sigma_{yy}}{\sigma_{xx}\sigma_{yy}}},\\&=-(s_h+s_y)\frac{\rho_{hy}}{\rho_{xy}}\sqrt{\frac{\sigma_{hh}}{\sigma_{xx}}},
\end{split} \end{equation}
and Eq.~$(iv)$ is satisfied if and only if
\begin{equation}     \begin{split}
    \gamma&= -\frac{(s_y+s_x)\sigma_{xy}}{\sigma_{yy}},\\&\overset{(a)}{=} -(s_y+s_x)\rho_{xy}\frac{\sqrt{\sigma_{xx}\sigma_{yy}}}{\sigma_{yy}},\\&=-(s_y+s_x)\rho_{xy}\sqrt{\frac{\sigma_{xx}}{\sigma_{yy}}},
\end{split} \end{equation}
where $(a)$, in both, $\sigma_{ij}=\rho_{ij}\sqrt{\sigma_{ii}\sigma_{jj}}$.

Since $d_x>0$ and can be chosen arbitrarily, Eq.~$(i)$ is satisfied if and only if 
\begin{equation}\label{eq:proof_three_loop_i}     \begin{split}
    0&>\gamma\sigma_{xy} + s_x\sigma_{xx} ,\\
    0&\overset{(a)}{>}-(s_y+s_x)\rho_{xy}\sqrt{\frac{\sigma_{xx}}{\sigma_{yy}}}\rho_{xy}\sqrt{\sigma_{xx}\sigma_{yy}} + s_x\sigma_{xx} ,\\
     0&>-(s_y+s_x)\rho_{xy}^2\sigma_{xx}+ s_x\sigma_{xx} ,\\
     0&>-(s_y+s_x)\rho_{xy}^2+ s_x ,\\
     \rho_{xy}^2s_y&>\big(1-\rho_{xy}^2\big) s_x ,\\
    \frac{\rho_{xy}^2}{1-\rho_{xy}^2}s_y&\overset{(b)}{>}s_x,
\end{split} \end{equation}
where $(a)$ $\rho_{ij}\sqrt{\sigma_{ii}\sigma_{jj}}$ and we substitute $\gamma$ and $(b)$ $\rho_{xy}^2<1$ such that $1-\rho^2_{xy}>0$ and the division doesn't flip the inequality. In addition, since  $1-\rho_{xy}^2>0$ and $s_y<0$, $\rho_{xy}^2/\big(1-\rho_{xy}^2\big)s_y<0$. Therefore $s_x<0$. Hence, let 
\begin{equation}
    s_x=b_1 \frac{\rho_{xy}^2}{1-\rho_{xy}^2}s_y,
\end{equation}
with $b_1>1$.

Since $d_h>0$ and can be chosen arbitrarily, Eq.~$(iii)$ is satisfied if and only if
\begin{equation}     \begin{split}
    0&>\beta \sigma_{xh} + s_{h} \sigma_{hh} ,\\
    0&\overset{(a)}{>}-(s_h+s_y)\frac{\rho_{hy}}{\rho_{xy}}\sqrt{\frac{\sigma_{hh}}{\sigma_{xx}}}\rho_{hx}\sqrt{\sigma_{xx}\sigma_{hh}} + s_{h} \sigma_{hh} ,\\
    0&>-(s_h+s_y)\frac{\rho_{hy}\rho_{hx}}{\rho_{xy}}\sigma_{hh}+ s_{h} \sigma_{hh} ,\\
     0&>-(s_h+s_y)\frac{\rho_{hy}\rho_{hx}}{\rho_{xy}}+ s_{h} ,\\
     \frac{\rho_{hy}\rho_{hx}}{\rho_{xy}}s_y&>\big(1-\frac{\rho_{hy}\rho_{hx}}{\rho_{xy}}\big) s_{h},
\end{split} \end{equation}
where $(a)$ $\rho_{ij}\sqrt{\sigma_{ii}\sigma_{jj}}$ and we substitute $\beta$. 

Note that $\rho_{hx}\rho_{hy}/\rho_{xy}$  is unconstrained except for the requirement that $\Sigma \in F_g$, which implies $\sigma_{ij} \neq 0$ and hence $\rho_{hx}\rho_{hy}/\rho_{xy} \neq 0$. Let $d:=\rho_{hx}\rho_{hy}/\rho_{xy}$. We have four scenarios:
\begin{enumerate}
    \item If $d<0$, then \begin{equation}     \begin{split}
        s_h&\overset{(a)}{<}\frac{d}{1-d}s_y ,\\
        s_h&\overset{(b)}{=}b_2\frac{d}{1-d}s_y \qquad \text{with }\, b_2<1
    \end{split} \end{equation}
    \item if $0<d<1$, then \begin{equation}     \begin{split}
        s_h&\overset{(a)}{<}\frac{d}{1-d}s_y ,\\
        s_h&\overset{(c)}{=}b_2\frac{d}{1-d}s_y\qquad \text{with }\, b_2>1
    \end{split} \end{equation}
    \item if $d=1$, then \begin{equation}
        0\cdot s_h=0<s_y\overset{(d)}{<}0,
    \end{equation} 
    which is a contradiction.
    \item if $d>1$, then \begin{equation}     \begin{split}
        s_h&\overset{(e)}{>}\frac{d}{1-d}s_y ,\\
        s_h&\overset{(f)}{=}b_2\frac{d}{1-d}s_y \qquad \text{with }\, b_2>1,
    \end{split} \end{equation}
\end{enumerate}
where $(a)$ due to $d<1$, we have that $1-d>0$ and the inequality isn't flipped, $(b)$ due to $d\leq0$, \,$1-d>0$ and $s_y<0$, we have that \,$d/\big(1-d\big)s_y>0$ such that $s_h$ being smaller than $d\big(1-d\big)s_y$ requires $b_2<1$, \,$(c)$ due to $d>0$,\, $1-d>0$ and $s_y<0$, we have that \,$d/\big(1-d\big)s_y<0$  such that $s_h$ being smaller than $d\big(1-d\big)s_y$ requires $b_2>1$, \,$(d)$ we use $s_y<0$,\, $(e)$ due to $d>1$, we have that\, $1-d<0$ such that the inequality is flipped and $(f)$ due to $d>0$, \,$1-d<0$ and $s_y<0$, we have that\, $d/\big(1-d\big)s_y>0$ such that $s_h$ being bigger than $d/\big(1-d\big)s_y$ requires $b_2>1$. Summarised this gives,
\begin{equation}
    s_h=b_2\frac{d}{1-d}s_y, \qquad \text{with}\, \begin{cases}
b_2<1 & \text{if } d\leq 0,\\
b_2>1 & \text{if } d>0.
\end{cases}
\end{equation}
In addition we can write $s_h$ as
\begin{equation}
    \begin{split}
        s_h&=b_2\frac{d}{1-d}s_y,\\
        &\overset{(a)}{=}b_2\frac{\rho_{hx}\rho_{hy}/\rho_{xy}}{1-\rho_{hx}\rho_{hy}/\rho_{xy}}s_y,\\
        &=b_2\frac{\rho_{hx}\rho_{hy}}{\rho_{xy}-\rho_{hx}\rho_{hy}}s_y,
    \end{split}
\end{equation}
where $(a)$ substitute $d=\rho_{hx}\rho_{hy}/\rho_{xy}$.

Using all of the above in Eq. $(ii)$, we get
\begin{equation}     \begin{split}
    0&=\beta\sigma_{xx} + \gamma\sigma_{hy} + (s_h+s_x)\sigma_{xh} ,\\
    0&\overset{(a)}{=}-(s_h+s_y)\frac{\rho_{hy}}{\rho_{xy}}\sqrt{\frac{\sigma_{hh}}{\sigma_{xx}}}\sigma_{xx}-(s_y+s_x)\rho_{xy}\sqrt{\frac{\sigma_{xx}}{\sigma_{yy}}}\rho_{hy}\sqrt{\sigma_{hh}\sigma_{yy}} + (s_h+s_x)\rho_{xh}\sqrt{\sigma_{xx}\sigma_{hh}} ,\\
    0&=-(s_h+s_y)\frac{\rho_{hy}}{\rho_{xy}}\sqrt{\sigma_{hh}\sigma_{xx}} -(s_y+s_x)\rho_{xy}\rho_{xy}\sqrt{\sigma_{xx}\sigma_{hh}} + (s_h+s_x)\rho_{xh}\sqrt{\sigma_{xx}\sigma_{hh}} ,\\
    0&=-(s_h+s_y)\frac{\rho_{hy}}{\rho_{xy}} -(s_y+s_x)\rho_{xy}\rho_{xy} + (s_h+s_x)\rho_{xh} ,\\
     0&\overset{(b)}{=}-(b_2\frac{\rho_{hx}\rho_{hy}}{\rho_{xy}-\rho_{hx}\rho_{hy}}s_y+s_y)\frac{\rho_{hy}}{\rho_{xy}} -(s_y+b_1 \frac{\rho_{xy}^2}{1-\rho_{xy}^2}s_y)\rho_{xy}\rho_{xy} + (b_2\frac{\rho_{hx}\rho_{hy}}{\rho_{xy}-\rho_{hx}\rho_{hy}}s_y+b_1 \frac{\rho_{xy}^2}{1-\rho_{xy}^2}s_y)\rho_{xh} ,\\
     0&=\left[b_1\frac{\rho_{xy}^2}{1-\rho_{xy}^2}\left(\rho_{hx}-\rho_{hy}\rho_{xy}\right)+b_2\frac{\rho_{xy}\rho_{hx}}{\rho_{xy}-\rho_{hy}\rho_{hx}}\left(\rho_{hx}-\frac{\rho_{hy}}{\rho_{xy}}\right) -\left(\frac{\rho_{hy}}{\rho_{xy}}+\rho_{xy}\rho_{hy}\right)\right]s_y ,\\
    0&=b_1\frac{\rho_{xy}^2}{1-\rho_{xy}^2}\left(\rho_{hx}-\rho_{hy}\rho_{xy}\right)+b_2\frac{\rho_{xy}\rho_{hx}}{\rho_{xy}-\rho_{hy}\rho_{hx}}\left(\rho_{hx}-\frac{\rho_{hy}}{\rho_{xy}}\right) -\left(\frac{\rho_{hy}}{\rho_{xy}}+\rho_{xy}\rho_{hy}\right) ,\\
    0&\overset{(c)}{=} b_1a + b_2 b -c,
\end{split} \end{equation}
where $(a)$ substitute the expression for $\beta$, $\gamma$ and $\sigma_{ij}=\rho_{ij}\sqrt{\sigma_{ii}\sigma_{jj}}$, $(b)$ substitute $s_x$ and $s_h$ and $(c)$ define $a:=\frac{\rho_{xy}^2}{1-\rho_{xy}^2}\left(\rho_{hx}-\rho_{hy}\rho_{xy}\right)$, $b:=\frac{\rho_{xy}\rho_{hx}}{\rho_{xy}-\rho_{hy}\rho_{hx}}\left(\rho_{hx}-\frac{\rho_{hy}}{\rho_{xy}}\right)$ and $c:=\frac{\rho_{hy}}{\rho_{xy}}+\rho_{xy}\rho_{hy}$. 

Note that $a,b$ and $c$ are unconstrained while $b_2$ depends on the sign of $d=\rho_{hx}\rho_{hy}/\rho_{xy}$. Since the sign of $c$ doesn't matter, this gives us in total 8 outcomes to check. We provide the derivation of two possible outcomes. The other outcomes follow analogously. Let $a,b<0$, then
\begin{equation}     \begin{split}
    0&= b_1a + b_2 b -c,\\
    -b_1a&=b_2 b -c,\\
    -a&\overset{(a)}{<}b_2 b -c,\\
    -a+c&<b_2 b,\\
    \frac{-a+c}{b}&\overset{(b)}{>}b_2,
\end{split} \end{equation}
where $(a)$ $b_1>1$ and $a<0$ such that $-b_1a>-a$ and $(b)$ $b<0$ such that the inequality flips.
If $d<0$, then $b_2<1$, such that $b_2$ is unconstrained from below and is always possible. If $d>0$ we have that $b_2>1$, such that $(-a+c)/b>b_2>1$ if and only if $(-a+c)/b>1$. Therefore if $a,b<0$ and $d>0$ we have a contradiction when $(-a+c)/b\leq1$.

Listing all the conditions that lead to a contradiction, we obtain:
\begin{equation}\label{eq:proof_three_loop_conditions}
    \begin{split}
        (c.1)&\quad \text{If $d>0,\,a<0$ and $b<0$, and $(-a+c)/b\leq1$}  ,\\
        (c.2)&\quad  \text{If $d>0,\,a>0$ and $b>0$, and $(-a+c)/b\leq1$} ,\\
        (c.3)&\quad  \text{If $d<0,\,a<0$ and $b>0$, and $(-a+c)/b\geq1$} ,\\
        (c.4)&\quad \text{If $d<0,\,a>0$ and $b<0$, and $(-a+c)/b\geq1$},\\
        (c.5)&\quad d=1.
    \end{split}
\end{equation}
Hence, a covariance matrix $\Sigma$ does not belongs to $\mathcal{M}^0_{G,\alpha}$ if and only if it satisfies one of the conditions in Eq.~\eqref{eq:proof_three_loop_conditions}. By the $\mathcal{M}^0$-criterion Theorem~\ref{thm:m0-criterion}, we have that $\Sigma\in \mathcal{M}^p_{G,\alpha}$ satisfying one of the conditions in Eq.~\eqref{eq:proof_three_loop_conditions} are identifiable. 

\end{proof}
\subsubsection{Instrumental Variable (IV)}\label{sec:proof_no_latent_cyclic}
\begin{proof}
     We adopt the assumptions and conventions stated at the start of this section. Let $G=(V,E)$ be the graph of Fig.~\ref{fig:iv}. The nodes $V=\{H,X,Y\}$ correspond to the SDE process $X=(H,X,Y)^T$, then the Hurwitz stable drift matrix $A$ respecting the causal structure of graph $G$ is
\begin{equation}
    A=\left[\begin{matrix}s_{z} & 0 & 0 & 0,\\0 & s_{h} & 0 & 0,\\\beta & \gamma & s_{x} & 0,\\0 & \delta & \alpha & s_{y}\end{matrix}\right],
\end{equation}

the diagonal drift matrix is
\begin{equation}
    D=\left[\begin{matrix}d_{z} & 0 & 0 & 0,\\0 & d_{h} & 0 & 0,\\0& 0 & d_{x} & 0,\\0 & 0 & 0 & d_{y}\end{matrix}\right]\in PDD_4,
\end{equation}

 and the $t$-faithful covariance matrix is 
\begin{equation}
    \Sigma=\left[\begin{matrix}\sigma_{zz} & 0 & \sigma_{zx} & \sigma_{zy},\\0& \sigma_{hh} & \sigma_{hx} & \sigma_{hy},\\\sigma_{zx} & \sigma_{hx} & \sigma_{xx} & \sigma_{xy},\\\sigma_{zy} & \sigma_{hy} & \sigma_{xy} & \sigma_{yy}\end{matrix}\right] \in\mathcal{M}^p_{G,\alpha}.
\end{equation}

The resulting set of equations to solve is
\begin{equation}     \begin{split}
   (i);& - d_{z} = 2 s_{z} \sigma_{zz} ,\\
    (ii);& - d_{h} = 2 s_{h} \sigma_{hh} ,\\
    (iii);& 0 = \beta \sigma_{zz} + s_{x} \sigma_{zx} + s_{z} \sigma_{zx} ,\\
    (iv);& 0 = \gamma \sigma_{hh} + s_{x} \sigma_{hx} + s_{h} \sigma_{hx} ,\\
    (v);& - d_{x} = 2 \beta \sigma_{zx} + 2 \gamma \sigma_{hx} + 2 s_{x} \sigma_{xx} ,\\
    (vi);& 0 = \alpha \sigma_{zx} + s_{y} \sigma_{zy} + s_{z} \sigma_{zy} ,\\
    (vii);& 0 = \alpha \sigma_{hx} + \delta \sigma_{hh} + s_{y} \sigma_{hy} + s_{h} \sigma_{hy} ,\\
    (viii);&0 = \alpha \sigma_{xx} + \beta \sigma_{zy} + \delta \sigma_{hx} + \gamma \sigma_{hy} + s_{x} \sigma_{xy} + s_{y} \sigma_{xy}  ,\\
    (ix);&- d_{y} = 2 \alpha \sigma_{xy} + 2 \delta \sigma_{hy} + 2 s_{y} \sigma_{yy}.
\end{split} \end{equation}

Analogous to the proof in~\ref{sec:proof_no_laten_cause_n_effect} we see that  Eq.~$(vi)$ is satisfied if and only if  $\operatorname{sign}(\alpha)=\operatorname{sign}(\sigma_{zy})/\operatorname{sign}(\sigma_{zx})$. Since $\Sigma\in \mathcal{M}^p_{G,\alpha}$,  we have $\sigma_{zy}\neq0$ and $\sigma_{zx}\neq0$  meaning that the sign of $\alpha$ is $+$ or $-$. Therefore there exists no $\Sigma\in\mathcal{M}^0_{G,\alpha}$, such that by virtue of the $\mathcal{M}^0$ criterion Theorem~\ref{thm:m0-criterion}, for any $\Sigma\in \mathcal{M}^p_{G,\alpha}$, 
the sign of edge $\alpha$ in graph $G$ is identifiable.
\end{proof}

\subsubsection{Cycle with IV}\label{sec:proof_no_latent_cyclic_iv}
\begin{proof}
     We adopt the assumptions and conventions stated at the start of this section. Let $G=(V,E)$ be the graph of Fig.~\ref{fig:iv_cyclic}. The nodes $V=\{H,X,Y\}$ correspond to the SDE process $X=(H,X,Y)^T$, then the Hurwitz stable drift matrix $A$ respecting the causal structure of graph $G$ is
\begin{equation}
    A=\left[\begin{matrix}s_{z} & 0 & 0 & 0,\\0 & s_{h} & 0 & \delta,\\\beta & \gamma & s_{x} & 0,\\0 & 0 & \alpha & s_{y}\end{matrix}\right],
\end{equation}
the diagonal drift matrix is
\begin{equation}
    D=\left[\begin{matrix}d_{z} & 0 & 0 & 0,\\0 & d_{h} & 0 & 0,\\0& 0 & d_{x} & 0,\\0 & 0 & 0 & d_{y}\end{matrix}\right]\in PDD_4,
\end{equation}
and the $t$-faithful covariance matrix is 
\begin{equation}
    \Sigma=\left[\begin{matrix}\sigma_{zz} & \sigma_{hz} & \sigma_{zx} & \sigma_{zy},\\\sigma_{hz} & \sigma_{hh} & \sigma_{hx} & \sigma_{hy},\\\sigma_{zx} & \sigma_{hx} & \sigma_{xx} & \sigma_{xy},\\\sigma_{zy} & \sigma_{hy} & \sigma_{xy} & \sigma_{yy}\end{matrix}\right]\in\mathcal{M}^p_{G,\alpha}.
\end{equation}

The resulting set of equations to solve is
\begin{equation}     \begin{split}
   (i)&; - d_{z} = 2 s_{z} \sigma_{zz} ,\\
    (ii)&; 0 = \delta \sigma_{zy} + s_{h} \sigma_{hz} + s_{z} \sigma_{hz} ,\\
    (iii)&; - d_{h} = 2 \delta \sigma_{hy} + 2 s_{h} \sigma_{hh} ,\\
    (iv)&;0 = \beta \sigma_{zz} + \gamma \sigma_{hz} + s_{x} \sigma_{zx} + s_{z} \sigma_{zx} ,\\
    (v)&;0 = \beta \sigma_{hz} + \delta \sigma_{xy} + \gamma \sigma_{hh} + s_{h} \sigma_{hx} + s_{x} \sigma_{hx} ,\\
    (vi)&;- d_{x} = 2 \beta \sigma_{zx} + 2 \gamma \sigma_{hx} + 2 s_{x} \sigma_{xx} ,\\
    (vii)&; 0 = \alpha \sigma_{zx} + s_{y} \sigma_{zy} + s_{z} \sigma_{zy}  ,\\
    (viii)&;0 = \alpha \sigma_{hx} + \delta \sigma_{yy} + s_{h} \sigma_{hy} + s_{y} \sigma_{hy} ,\\
    (ix)&;0 = \alpha \sigma_{xx} + \beta \sigma_{zy} + \gamma \sigma_{hy} + s_{x} \sigma_{xy} + s_{y} \sigma_{xy} ,\\
    (x)&;- d_{y} = 2 \alpha \sigma_{xy} + 2 s_{y} \sigma_{yy}.
\end{split} \end{equation}

We note that since the drift matrix $A$ is not triangular, the self loops $s_x,s_y$ and $s_h$ are not constrained from the outset of the proof. 

Due to $d_z,\sigma_{zz}>0$, Eq.~$(i)$ is satisfied if and only if $s_z<0$. In addition, Eq.~$(vii)$ is satisfied if and only if
\begin{equation}
    \alpha=-\big(s_y+s_z\big)\frac{\sigma_{zy}}{\sigma_{zx}}.
\end{equation}

Since $d_y>0$, Eq.~$(x)$ is satisfied if and only if
\begin{equation}
    \begin{split}
        0&>2\alpha\sigma_{xy}+2s_y\sigma_{yy},\\
        0&>\alpha\sigma_{xy}+s_y\sigma_{yy},\\
        0&\overset{(a)}{>}-\big(s_y+s_z\big)\frac{\sigma_{zy}}{\sigma_{zx}}\sigma_{xy}+s_y\sigma_{yy},\\
         0&\overset{(b)}{>}-\big(s_y+s_z\big)\frac{\rho_{zy}\sqrt{\sigma_{zz}\sigma_{yy}}\rho_{xy}\sqrt{\sigma_{xx}\sigma_{yy}}}{\rho_{zx}\sqrt{\sigma_{zz}\sigma_{x}}}+s_y\sigma_{yy},\\
         0&>-\big(s_y+s_z\big)\frac{\rho_{zy}\rho_{xy}}{\rho_{zx}}\sigma_{yy}+s_y\sigma_{yy},\\
         0&>-\big(s_y+s_z\big)\frac{\rho_{zy}\rho_{xy}}{\rho_{zx}}+s_y,\\
         \frac{\rho_{zy}\rho_{xy}}{\rho_{zx}}s_z&>\big(1-\frac{\rho_{zy}\rho_{xy}}{\rho_{zx}}\big)s_y.
    \end{split}
\end{equation}
where $(a)$ we substituted $\alpha=\big(s_y+s_z\big)\sigma_{zy}/\sigma_{zx}$ and $(b)$ we substituted $\sigma_{ij}=\rho_{ij}\sqrt{\Sigma_{ii}\Sigma_{jj}}$. Let $d=\rho_{zy}\rho_{xy}/\rho_{zx}$ for a simpler notation. Then the inequality can be written as
\begin{equation}
    ds_z>\big(1-d\big)s_y.
\end{equation}

We have five scenarios
\begin{enumerate}
    \item If $d<0$, then 
    \begin{equation}
        s_z\frac{d}{(1-d)}>s_y.
    \end{equation}
    Since $d<0$,\,  $d/(1-d)<0$. Therefore $s_y<a$, where $a>0$, such that $y$ can be both positive and negative. Since $d<0$, \, $|1-d|>|d|$ and thus $|d/(1-d)|<1$. Therefore, if $s_y>0$,
    \begin{equation}
        \begin{split}
            |s_y|&<|s_z\frac{d}{(1-d)}|,\\
            &=|s_z|\cdot|\frac{d}{(1-d)}|,\\
             &<|s_z|.
        \end{split}
    \end{equation}
    In addition, $s_z<0$, therefore $\operatorname{sign}(s_z+s_y)=\operatorname{sign}(s_z)=-$. If $s_y<0$, then $\operatorname{sign}(s_z+s_y)=-$
    \item If $d=0$, then
    \begin{equation}
        0>s_y.
    \end{equation}
    In addition, $s_z<0$, therefore $\operatorname{sign}(s_z+s_y)=-$.
    \item  If $0<d<1$, then
    \begin{equation}
        s_z\frac{d}{(1-d)}>s_y.
    \end{equation}
    Since $0<d<1$,\, $d/(1-d)>0$. Moreover, $s_z<0$, therefore $s_zd/(1-d)<0$ and thus $s_y<0$. Hence, $\operatorname{sign}(s_z+s_y)=-$.
    \item If $d=1$, then
    \begin{equation}
        s_z>0.
    \end{equation}
    Since $s_z<0$, this is a contradiction.
    \item  If $d>1$, then
    \begin{equation}
        s_z\frac{d}{(1-d)}<s_y,
    \end{equation}
    since $1-d<0$ the sign of the inequality has flipped. In addition, since $d/(1-d)<0$ and $s_z<0$, \, $s_z\frac{d}{(1-d)}>0$ and therefore $s_y>0$. Furthermore, since $d>0$, \, $|d|>|1-d|$ and thus $|d/(1-d)|>1$. Hence, 
    \begin{equation}
        \begin{split}
            |s_y|&>|s_z\frac{d}{(1-d)}|,\\
        &=|s_z|\cdot|\frac{d}{(1-d)}|,\\
        &>|s_z|.
        \end{split}
    \end{equation}
    Therefore, $\operatorname{sign}(s_z+s_y)=\operatorname{sign}(s_y)=+$.
\end{enumerate}

To summarise the result, 
\begin{equation}
    \operatorname{sign}(s_z+s_y)=\begin{cases}
        -\quad \text{, if } d<1,\\
        +\quad \text{, if } d>1,
    \end{cases}
\end{equation}
and if $d=1$ there is no valid solution for the set of equations and hence $\Sigma\not\in \mathcal{M}^p_{G,\alpha}$. Therefore the sign of $\alpha=-\big(s_y+s_z\big)\sigma_{zy}/\sigma_{zx}$ is
\begin{equation}
    \operatorname{sign}(\alpha)=\begin{cases}
        \operatorname{sign}(\sigma_{zy}/\sigma_{zx})\qquad \text{, if }\, \rho_{zy}\rho_{xy}/\rho_{zx}<1,\\
        -\operatorname{sign}(\sigma_{zy}/\sigma_{zx})\quad \text{, if }\, \rho_{zy}\rho_{xy}/\rho_{zx}>1.
    \end{cases}
\end{equation}

Since $\Sigma\in \mathcal{M}^p_{G,\alpha}$, we have $\sigma_{zy}\neq0$ meaning that the sign of $\alpha$ is $+$ or $-$. Therefore there exists no $\Sigma\in\mathcal{M}^0_{G,\alpha}$, such that by virtue of the $\mathcal{M}^0$ criterion Theorem~\ref{thm:m0-criterion}, for any $\Sigma\in \mathcal{M}^p_{G,\alpha}$, 
the sign of edge $\alpha$ in graph $G$ is identifiable.

\end{proof}

\subsection{Theorem~\ref{thm:latent_id}}\label{sec:proof_latent_id}
We use results and steps from the proofs for that same structures without latent variables detailed in~\ref{sec:proof_no_latent}. The only difference in the setup will be that the variable $H$ is hidden now. This means that the covariance matrix values that depend on $H$, i.e., $\Sigma_{h.}$, will be unknown. Therefore these variables are treated as unknown variables and can be chosen within the bounds of $\Sigma\in F_G$.

\subsubsection{Cause and Effect}\label{sec:appendix_latent_cause_n_effect}
\begin{proof}
We adopt the assumptions and conventions stated at the start of this section. Let $G=(V,E)$ be the graph of Fig.~\ref{fig:hy}, where node $H$ is hidden. By Lemma~\ref{lem:no_latent_id_values}, $\operatorname{sign}(\alpha) = \operatorname{sign}(\sigma_{hy})$. Since we can choose $\sigma_{hy}$, the sign of $\alpha$ can always be chosen both positive and negative. Thus for any $\Sigma\in\mathcal{M}^p_{G,\alpha}$, we have $\Sigma\in\mathcal{M}^+_{G,\alpha}$ and $\Sigma\in\mathcal{M}^-_{G,\alpha}$. Therefore, $\mathcal{M}^+_{G,\alpha}=\mathcal{M}^-_{G,\alpha}$, using Definition~\ref{def:edge_identifiability}, $\alpha$ is non-identifiable in graph $G$
\end{proof}

\subsubsection{Confounding}\label{app:proof_latent_confounding}
\begin{proof}
     We adopt the assumptions and conventions stated at the start of this section. Let $G=(V,E)$ be the graph of Fig.~\ref{fig:confounding}, where node $H$ is hidden. By Lemma~\ref{lem:no_latent_part_id_values}, any covariance matrix $\Sigma$ satisfying one of the the following conditions renders $\alpha$ identifiable:
    \begin{equation}
            \begin{split}
            (c.1)&\quad  \frac{(2\rho_{hy}^2\rho_{hx}^2-\rho_{hy}^2-\rho_{hx}^2)(\rho_{hx}\rho_{hy}-\rho_{xy})}{2\rho_{hx}\rho_{hy}(\rho_{hx}^2-1)(\rho_{hy}^2-1)}\geq1,\\
            (c.2)&\quad\,\,\operatorname{sign}(\sigma_{hx}\sigma_{hy})\neq\operatorname{sign}(\sigma_{xy}),
            \\
    (c.3)&\quad \frac{\rho_{hy}\rho_{hx}}{\rho_{xy}}\geq1.
                    \end{split}
    \end{equation}

    Hence, a covariance matrix $\Sigma$ satisfying all the following conditions renders $\alpha$ non-identifiable:
    \begin{equation}\label{eq:proof_latent_confounding_conditions}
    \begin{split}
    (c'.1)&\qquad \frac{(2\rho_{hy}^2\rho_{hx}^2-\rho_{hy}^2-\rho_{hx}^2)(\rho_{hx}\rho_{hy}-\rho_{xy})}{2\rho_{hx}\rho_{hy}(\rho_{hx}^2-1)(\rho_{hy}^2-1)}<1, \\
    (c'.2)&\quad \operatorname{sign}(\sigma_{hx}\sigma_{hy})=\operatorname{sign}(\sigma_{xy}), \\
    (c'.3)&\quad \frac{\rho_{hy}\rho_{hx}}{\rho_{xy}}<1.
    \end{split}
\end{equation}

We will try to construct such a covariance matrix $\Sigma$ satisfying the conditions in Eq.~\eqref{eq:proof_latent_confounding_conditions} by choosing the values of the hidden part of the covariance matrix, i.e., in
\begin{equation}
    \Sigma=\left[\begin{matrix}\sigma_{hh} & \sigma_{hx} & \sigma_{hy},\\ \sigma_{hx} & \sigma_{xx} & \sigma_{xy},\\\sigma_{hy} & \sigma_{xy} & \sigma_{yy}\end{matrix}\right]\in\mathcal{M}^p_{G,\alpha}
\end{equation}
we set the values for
\begin{equation}
    \left[\begin{matrix}\sigma_{hh} & \sigma_{hx} & \sigma_{hy},\\\sigma_{hx} & . & .,\\\sigma_{hy} & . & .\end{matrix}\right]\in\mathcal{M}^p_{G,\alpha}.
\end{equation}

We set $\rho_{hx}=a|\rho_{xy}|$ and $\rho_{hy}=a\rho_{xy}$ with $0<a<1$.

To begin with Condition~$(c'.1)$, this becomes
\begin{equation}\label{eq:proof_latent_confouding_step1}
    \begin{split}
         \frac{(2\rho_{hy}^2\rho_{hx}^2-\rho_{hy}^2-\rho_{hx}^2)(\rho_{hx}\rho_{hy}-\rho_{xy})}{2\rho_{hx}\rho_{hy}(\rho_{hx}^2-1)(\rho_{hy}^2-1)}&<1,\\
         \frac{(2a^4\rho_{xy}^4-a^2\rho_{xy}^2-a^2\rho_{xy}^2)(a^2|\rho_{xy}|-1)}
        {2a^2|\rho_{xy}|(a^2\rho_{xy}^2-1)(a^2\rho_{xy}^2-1)}&<1,\\
        (2a^4\rho_{xy}^4-a^2\rho_{xy}^2-a^2\rho_{xy}^2)(a^2|\rho_{xy}|-1)
        &\overset{(a)}{<}2a^2|\rho_{xy}|(a^2\rho_{xy}^2-1)(a^2\rho_{xy}^2-1),
    \end{split}
\end{equation}
where in $(a)$ we use that the denominator is positive (see below Eq.~\eqref{eq:proof_confounding_reference_latent}). Since Eq.~\eqref{eq:proof_latent_confouding_step1} contains only contributions $\rho_{xy}^2$ and $|\rho_{xy}|$, we can without loss of generality only consider the domain $\rho_{xy}\in(0,1)$ for the derivation. Then Eq.~\eqref{eq:proof_latent_confouding_step1} becomes
\begin{equation}\label{eq:proof_latent_confounding_step2}
    \begin{split}
        (2a^4\rho_{xy}^4-a^2\rho_{xy}^2-a^2\rho_{xy}^2)(a^2\rho_{xy}-1)
        &<2a^2\rho_{xy}(a^2\rho_{xy}^2-1)(a^2\rho_{xy}^2-1),\\
        2a^6\rho_{xy}^5-2a^4\rho_{xy}^3-2a^4\rho_{xy}^4+2a^2\rho_{xy}^2
-2a^2\rho_{xy}(a^2\rho_{xy}^2-1)^2&<0,\\
2a^6\rho_{xy}^5-2a^4\rho_{xy}^3-2a^4\rho_{xy}^4+2a^2\rho_{xy}^2
-2a^2\rho_{xy}(a^4\rho_{xy}^4-2a^2\rho_{xy}^2+1)&<0,\\
2a^6\rho_{xy}^4-2a^4\rho_{xy}^2-2a^4\rho_{xy}^3+2a^2\rho_{xy}
-2a^2(a^4\rho_{xy}^4-2a^2\rho_{xy}^2+1)&<0,\\
2a^6\rho_{xy}^4-2a^4\rho_{xy}^2-2a^4\rho_{xy}^3+2a^2\rho_{xy}^2
-2a^6\rho_{xy}^4+4a^4\rho_{xy}^2-2a^2&<0,\\
-2a^4\rho_{xy}^3+2(a^2+a^4)\rho_{xy}^2-2a^2&<0.
    \end{split}
\end{equation}

Figure~\ref{fig:proof_latent_confounding_c1} shows the shape of the left-hand side of Eq.~\eqref{eq:proof_latent_confounding_step2}. The shape is symmetric around 0, and on the domain $\rho_{xy}\in[0,2)$ the shape is convex. Therefore, on the domain $\rho_{xy}\in(0,1)$,
\begin{equation}\label{eq:proof_latent_confounding_step3}
-2a^4\rho_{xy}^3+2(a^2+a^4)\rho_{xy}^2-2a^2<-2a^41^3+2(a^2+a^4)1^2-2a^2=0.
\end{equation}
Hence, Condition~$(c'.1)$ is satisfied.

To continue with Condition~$(c'.2)$. We have that
\begin{equation}
    \operatorname{sign}(\sigma_{hx}\sigma_{hy})=\operatorname{sign}(a|\rho_{xy}|a\rho_{xy})=\operatorname{sign}(\rho_{xy})=\operatorname{sign}(\sigma_{xy}),
\end{equation}
where we have used that $e,|\rho_{xy}|>0$. Hence, Condition~$(c'.2)$ is satisfied.

Finally, we consider Condition~$(c'.3)$. We have that
\begin{equation}
    \frac{\rho_{hy}\rho_{hx}}{\rho_{xy}}=\frac{a|\rho_{xy}|a\rho_{xy}}{\rho_{xy}}=a^2|\rho_{xy}|<1,
\end{equation}
where we have used that $0<a<1$ and $0<|\rho_{xy}|<1$. Hence, Condition~$(c'.3)$ is satisfied. Therefore, all conditions are satisfied and the choice $\rho_{hx}=a|\rho_{xy}|$ and $\rho_{hy}=a\rho_{xy}$ with $0<a<1$ makes the covariance matrix $\Sigma$ non-identifiable.

\begin{figure}
    \centering

    \begin{subfigure}[b]{0.4\textwidth}
        \centering
        \includegraphics[width=\textwidth]{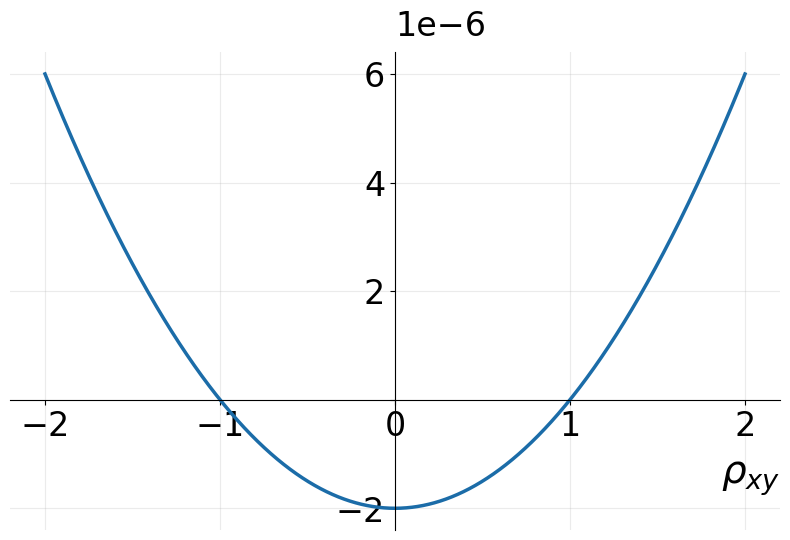}
        \caption{The left-hand side of Eq.~\eqref{eq:proof_latent_confounding_step2}.}
        \label{fig:proof_latent_confounding_c1}
    \end{subfigure}
    \hfill
    \begin{subfigure}[b]{0.4\textwidth}
        \centering
        \includegraphics[width=\textwidth]{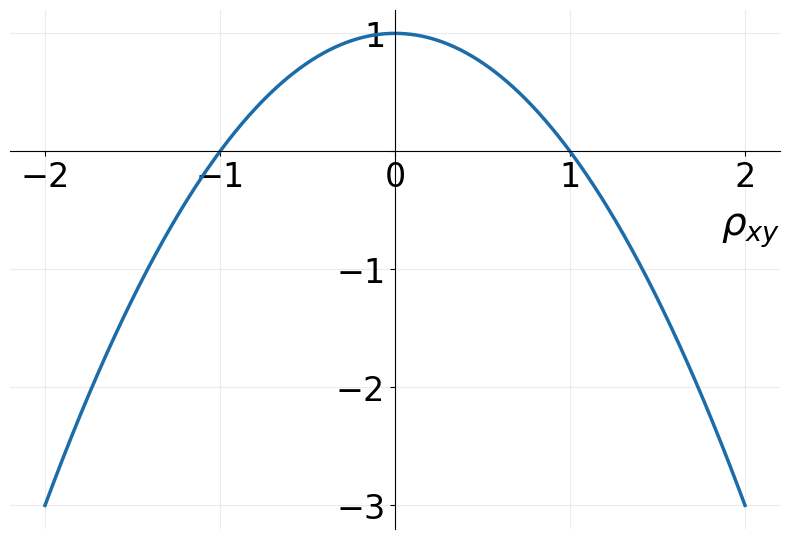}
        \caption{The right-hand side of Eq.~\eqref{eq:proof_latent_confounding_sylvester}.}
        \label{fig:proof_latent_confounding_pd3}
    \end{subfigure}

    \caption{The sub-figures show the shape of two functions that are used in the proof of Appendix~\ref{app:proof_latent_confounding}.}
    \label{fig:proof_latent_confounding}
\end{figure}

Next we verify if the choice is $t$-faithful.
We begin by verifying that our choice $\Sigma\in PD_3$. We note that, if we set $\rho_{hx}=a|\rho_{xy}|$ and $\rho_{hy}=a\rho_{xy}$ with $0<a<1$. Then, since, $\rho_{xy}\in (-1,1)\setminus\set{0}$, we have that  $\rho_{hx}\in (0,1)$ and $\rho_{hy}\in (-1,1)\setminus\set{0}$. Thereby, satisfying part of the constraints. Therefore, only verifying if Sylvester's criterion holds remains. This means we need to verify Eq.~\eqref{eq:sylvester_criterion_3__withtout_latent}, i.e., 
\begin{equation}\label{eq:proof_latent_confounding_sylvester}
    \begin{split}
        0&<1+2\rho_{xy}\rho_{hx}\rho_{hy}-\big(\rho_{xy}^2+\rho_{hy}^2+\rho_{hx}^2\big),\\
        0&\overset{(a)}{<}1+2a^2|\rho_{xy}|\rho_{xy}^2-(a^2\rho_{xy}^2+a^2\rho_{xy}^2+\rho_{xy}^2),\\
        0&<1+2a^2\rho_{xy}^2(|\rho_{xy}|-1)-\rho_{xy}^2,
    \end{split}
\end{equation}
where in $(a)$ we set $\rho_{hx}=a|\rho_{xy}|$ and $\rho_{hy}=a\rho_{xy}$. Figure~\eqref{fig:proof_latent_confounding_pd3} shows the shape of the right-hand side of Eq.~\eqref{eq:proof_latent_confounding_sylvester}. The shape is symmetric around $\rho_{xy}=0$, and on the domain $\rho_{xy}\in[0,2)$ is concave. Hence, for $\rho_{xy}\in(-1,1)\setminus\set{0}$ we have that
\begin{equation}
    1+2a^2\rho_{xy}^2(|\rho_{xy}|-1)-\rho_{xy}^2>1+2a^2(\pm 1)^2\big(|\pm 1|-1\big)-(\pm 1)^2=0.
\end{equation}
Hence, Sylvester's criterion holds.

In addition, since $\rho_{hx}\in (0,1)$ and $\rho_{hy}\in (-1,1)\setminus\set{0}$, the choice respects the marginal independences. 

Finally, combining that our choice $\Sigma\in PD_3$ and $\Sigma$ respects the marginal independences, we have that $\Sigma\in F_G$. Therefore, by the proof in Appendix~\ref{sec:appendix_proof_no_latent_confoundin} $\mathcal{M}^p_{G,\alpha}=F_G$, we have that for any observable block $\Sigma_{oo}$, we can always construct a valid covariance matrix $\Sigma$ such that $\alpha$ is non-identifiable in graph $G$.

\end{proof}

\subsubsection{Instrumental Variable (IV)}
\begin{proof}
    We adopt the assumptions and conventions stated at the start of this section. Let $G=(V,E)$ be the graph of Fig.~\ref{fig:iv}, where node $H$ is hidden. By Lemma~\ref{lem:no_latent_id_values}, $\operatorname{sign}(\alpha)=\operatorname{sign}(\sigma_{zy})/\operatorname{sign}(\sigma_{zx})$. This is still constrained by the observed part of $\Sigma$. Therefore, we can use the same conclusion as in the case without latent variables. By Theorem~\ref{thm:no_latent_id}, the sign of edge $\alpha$ in graph $G$ is identifiable.
\end{proof}

\subsubsection{Cycle with IV}
\begin{proof}
    We adopt the assumptions and conventions stated at the start of this section. Let $G=(V,E)$ be the graph of Fig.~\ref{fig:iv_cyclic}, where node $H$ is hidden. By Lemma~\ref{lem:no_latent_id_values},
    \begin{equation}
        (\alpha)=\begin{cases}
        \operatorname{sign}(\sigma_{zy}/\sigma_{zx})\qquad \text{, if }\, \rho_{zy}\rho_{xy}/\rho_{zx}<1,\\
        -\operatorname{sign}(\sigma_{zy}/\sigma_{zx})\quad \text{, if }\, \rho_{zy}\rho_{xy}/\rho_{zx}>1.
    \end{cases}
    \end{equation}
    This is still constrained by the observed part of $\Sigma$. Therefore, we can use the same conclusion as in the case without latent variables. By Theorem~\ref{thm:no_latent_id}, the sign of edge $\alpha$ in graph $G$ is identifiable.
\end{proof}

\end{document}